\documentclass{amsart}

\usepackage[all]{xy}   
\usepackage{hyperref} 

\usepackage[usenames,dvipsnames]{xcolor}

\theoremstyle{plain}
\newtheorem{lem}{Lemma}[section]
\newtheorem{cor}[lem]{Corollary}
\newtheorem{prop}[lem]{Proposition}
\newtheorem{thm}[lem]{Theorem}

\theoremstyle{definition}
\newtheorem{ex}[lem]{Example}
\newtheorem{rem}[lem]{Remark}
\newtheorem{dfn}[lem]{Definition}

\newcommand{\Z}{\mathbb{Z}}               
\newcommand{\Q}{\mathbb{Q}}              
\newcommand{\R}{\mathbb{R}}              
\newcommand{\LL}{\mathbb{L}}                

\newcommand{\DF}{\mathbb{D}_F}   
\newcommand{\DP}{\mathbb{D}_{F,P}}
\newcommand{\DQ}{\mathbb{D}_{F,Q}}
\newcommand{\DFd}{\DF^\star}          
\newcommand{\DPd}{\DP^\star}  
\newcommand{\DQd}{\DQ^\star}  

\newcommand{\Os}{\mathcal{O}}    
\newcommand{\OP}{\Os^P}         
\newcommand{\QOP}{{}^Q\OP}    

\newcommand{\Gg}{\mathcal{G}}    
\newcommand{\GP}{\Gg^P}      
\newcommand{\QGP}{{}^Q\GP}    

\newcommand{\Ff}{\mathcal{F}}   
\newcommand{\Ll}{\mathcal{L}}   

\DeclareMathOperator{\CH}{\mathrm{CH}}      
\DeclareMathOperator{\Hom}{\mathrm{Hom}}      
\DeclareMathOperator{\End}{\mathrm{End}}      
\DeclareMathOperator{\hh}{\mathtt{h}}             

\newcommand{\de}{\delta}                   
\newcommand{\al}{\alpha}
\newcommand{\be}{\beta}
\newcommand{\la}{\lambda}
\newcommand{\om}{\omega}

\begin{document}

\title[Oriented cohomology sheaves on double moment graphs]{Oriented cohomology sheaves on double moment graphs}

\author[R.~Devyatov]{Rostislav Devyatov}
\address[Rostislav Devyatov]{Department of Mathematical and Statistical Sciences, University of Alberta, 632 Central Academic Building, Edmonton, AB,  T6G 2G1, Canada}
\email{devyatov@ualberta.ca}
\urladdr{http://www.mccme.ru/~deviatov/}

\author[M.~Lanini]{Martina Lanini}
\address[Martina Lanini]{Dipartimento di Matematica, Universit\`a di Roma Tor Vergata, Via della Ricerca Scientifica, 00133 Rome, Italy}
\email{lanini@mat.uniroma2.it}
\urladdr{https://sites.google.com/site/martinalanini5/home}

\author[K.~Zainoulline]{Kirill Zainoulline}
\address[Kirill Zainoulline]{Department of Mathematics and Statistics, University of Ottawa, 150 Louis-Pasteur, Ottawa, ON, K1N 6N5, Canada}
\email{kirill@uottawa.ca}
\urladdr{http://mysite.science.uottawa.ca/kzaynull/}


\begin{abstract} 
In the present paper we extend the theory of sheaves on moment graphs due to Braden-MacPherson and Fiebig to the context of an arbitrary oriented equivariant cohomology $\hh$ (e.g. to algebraic cobordism). We introduce and investigate structure $\hh$-sheaves on double moment graphs to describe equivariant oriented cohomology of products of flag varieties. We show that in the case of a total flag variety $X$ of Dynkin type $A$ the space of global sections of the double structure $\hh$-sheaf also describes the endomorphism ring of the equivariant $\hh$-motive of $X$.
\end{abstract}

\maketitle

\tableofcontents


\section{Introduction}

In \cite{GKM} Goresky, Kottwitz  and MacPherson showed that equivariant cohomology of varieties equipped with an action of a torus $T$ can be completely described by a certain graph. We refer to \cite{GKM} for the precise assumptions and examples. We only observe that for varieties with a finite number of isolated $T$-fixed points and 1-dimensional orbits (e.g. for flag varieties and Schubert varieties) this graph arises as a 1-skeleton of the $T$-action: vertices are the fixed points and edges are closures of 1-dimensional orbits. Since the torus acts on a 1-dimensional orbit via a character (which is uniquely defined up to a sign and which corresponds to a chosen orientation of the orbit closure), one can label edges of the graph by the corresponding characters. The resulting labelled graph is called the \emph{moment graph}.

In the subsequent works hypotheses on the $T$-action have been relaxed \cite{GHZ, GH, G} and the moment graph techniques have been applied to other cohomology theories \cite{B, R(Kn), VV, K, GR}. If the $T$-variety admits a stratification by $T$-invariant cells, then the inclusion relation of cells induces an orientation of the edges. In \cite{BMP} Braden and MacPherson showed that the information contained in this oriented moment graph is also sufficient to compute the (equivariant) intersection cohomology.  To do this they introduced the theory of \emph{sheaves on moment graphs}. This important concept motivated a series of papers by Fiebig, where he developed and axiomatized sheaves on moment graphs and exploited Braden-MacPherson's construction to attack various representation theoretical problems (see, for example,~\cite{F2}).

Let $G$ be a split semisimple linear algebraic group over a field of characteristic zero containing a maximal torus $T$. In the present paper, we propose a theory of moment graphs and sheaves on it adapted to any $G$-equivariant \emph{generalized cohomology theory} $\hh_G$ (e.g. equivariant Chow ring~\cite{To}, equivariant $K$-theory or equivariant algebraic cobordism~\cite{HML}). The sheaves we deal with, which we call \emph{oriented cohomology sheaves} or {\em $\hh$-sheaves}, are constituted by collections of modules and maps of modules over the equivariant coefficient ring $\hh_G(pt)$, satisfying certain conditions which are imposed by the graph data. The previously considered sheaves on moment graphs are hence obtained as the special case in which $G=T$ and $\hh$ is the Chow theory.

Let $P$ be a standard parabolic subgroup of $G$ containing $T$. We first look at the classical \emph{parabolic moment graph} $\GP$ of the projective homogeneous variety $G/P$ (see Example~\ref{ex:parabolic}) and associate to it the {\em structure ${\hh}$-sheaf} $\OP$ (see Example~\ref{ex:parsheaf}). The main result of \cite{CZZ2} says that the space of global sections $\Gamma(\OP)$ of the structure $\hh$-sheaf $\OP$ describes the $T$-equivariant oriented cohomology $\hh_T(G/P)$. We then generalize it as follows:

Given two standard parabolic subgroups $P$ and $Q$, we introduce a \emph{double moment graph} $\QGP$ (Definition~\ref{dfn:double}), and its associated structure $\hh$-sheaf $\QOP$ (Example~\ref{prop:motsheaf}). Suppose $G$ acts diagonally on the product $X=G/Q\times G/P$. Consider the induced action of the Weyl group $W$ on the $T$-equivariant cohomology $\hh_T(X)$ (see Section~\ref{sec:orcohprod}) and let $\hh_T(X)^W$ be the invariant subring. Our main result (Theorem~\ref{thm:main}) then says
\begin{itemize}
\item[(A)] $\hh_T(X)^W= \Gamma(\OP)^{W_Q}$, where $W_Q$ is the Weyl group of the Levi-part of $Q$;
\item[(B)] $\hh_T(X)^W \subseteq \Gamma(\QOP)$, 
where the equality holds if the double moment graph $\QGP$ is closed (see Definition~\ref{dfn:closure}).
\end{itemize}
The invariant subring $\hh_T(X)^W$ is related to the $G$-equivariant cohomology via the forgetful map $\hh_G(X) \to \hh_T(X)^W$ which has been recently studied in the context of \emph{motivic decompositions} of (generic) flag varieties \cite{CM}, \cite{PS}, \cite{NPSZ}, \cite{CNZ}. Observe that in all these studies the cohomology rings were described using the \emph{Schubert basis} (classes of orbit closures). In the present paper we describe it for the first time using the basis of the space of global sections (localization basis). We believe that such a new description leads to better understanding of cohomology and equivariant motives of flag varieties (see e.g. the subsequent paper \cite{LZ} where this basis is used to study cohomology of folded root systems).

As for the latter, consider the category of $G$-equivariant $\hh$-motives (see \cite[\S2]{GV} and \cite{PS} for definitions and basic properties). Let  $\mathfrak{M}_G$ be its full additive subcategory generated by motives of $G/P$ for all standard parabolic subgroups $P$ of $G$. The morphisms in $\mathfrak{M}_G$ are given by classes in $\hh_G(X)$ and their composition is given by the so called correspondence  product. The forgetful map induces a functor from $\mathfrak{M}_G$ to the category of $W$-invariant equivariant motives $\mathfrak{M}_T^W$. The categories $\mathfrak{M}_G$ and $\mathfrak{M}_T^W$ are closely related to each other  (see \cite{CNZ}). Our main result then implies that the global sections $\Gamma(\OP)^{W_Q}$ and $\Gamma(\QOP)$ of the respective structure sheaves describe the  $\Hom$-spaces between $W$-invariant equivariant motives $[G/Q]$ and $[G/P]$ (see Section~\ref{sec:motives}). In particular, it allows us (i)~to interpret the correspondence product in $\mathfrak{M}_T^W$ in moment graphs terms (Proposition~\ref{prop:corr}); and (ii)~to realize the endomorphism ring of the $W$-invariant equivariant motive of the total flag variety of type $A$ as the space of global sections of the structure {\tt h}-sheaf on the total double moment graph (Corollary~\ref{cor:typeA}).

The paper is organized as follows. In section~\ref{sec:prel} we recall basic facts concerning equivariant oriented cohomology and its coefficient rings; we sketch the algebraic construction of the equivariant cohomology of a flag variety. In the next section~\ref{sec:orcohprod} we study cohomology of the products of flag varieties; using the K\"unneth isomorphism we describe its $W$-invariants in terms of duals of the formal affine Demazure algebras.  Section~\ref{sec:doublemom} is dedicated to (parabolic, double) moment graphs and its combinatorics; we discuss $W$-action on moment graphs and its closure. In section~\ref{sec:closedgr} we describe closed double moment graphs for all classical Dynkin types and the type $G_2$ (we put all the technical details and proofs for type $B$ in the appendix). In section~\ref{sec:globsec} we introduce structure sheaves and its global sections; we state and prove our main result (Theorem~\ref{thm:main}). In the last section~\ref{sec:motives} we discuss applications to equivariant motives.

\medskip

\paragraph{\it Acknowledgements} This project started during the workshop Beyond Toric Geometry that was held in May 2017 in Oaxaca, Mexico. M.L. and K.Z. are hence very grateful to the organizers of this workshop and to BIRS Oaxaca for hospitality. R.D. was partially supported by the Fields Institute for Research in Mathematical Sciences, Toronto, Canada. R.D. and K.Z. were partially supported by the NSERC Discovery grant RGPIN-2015-04469, Canada. M.L. acknowledges the MIUR Excellence Department Project awarded to the Department of Mathematics, University of Rome Tor Vergata, CUP E83C18000100006.


\section{Preliminaries}\label{sec:prel}

\paragraph{\it Formal group laws.}
Let $R$ be a commutative ring. Following \cite[p.4]{LM} consider a one-dimensional commutative formal group law $F$ over $R$ that is a power series in two variables $x+_F y=F(x,y)\in R[[x,y]]$ which satisfies $x+_F y=y+_F x$, $x+_F 0=x$ and $(x+_F y)+_F z=x+_F(y+_F z)$. It can be written as
\[
x+_F y=x+y+\sum_{i,j>0} a_{ij}x^i y^j,\;\text{ for some coefficients }a_{ij}\in R.
\]
There is the formal inverse $-_Fx$ that is the power series in one variable satisfying $x+_F(-_Fx)=0$. The polynomial ring in variables $a_{ij}$'s modulo relations imposed by commutativity and associativity gives the so called Lazard ring $\LL$. There is the universal formal group law $F_U$ over $\LL$ together with the evaluation map $F_U\to F$ (given by evaluating the coefficients).

\begin{ex}
Basic examples of formal group laws (f.g.l.) are  
\begin{itemize}
\item The additive f.g.l. $x+_F y=x+y$ over $R=\Z$ with $F_U\to F$ given by $a_{ij}\mapsto 0$ and $-_F x=-x$ \cite[1.1.4]{LM}.
\item The multiplicative f.g.l. $x+_F y=x+y-\be xy$ over $R=\Z[\be]$ with $F_U\to F$ given by $a_{11}\mapsto -\be$, $a_{ij}\mapsto 0$ for all $i,j>1$ and $-_F x=\tfrac{x}{\be x-1}$. If $\be$ is invertible and $R=\Z[\be^{\pm 1}]$, then the respective $F$ is called multiplicative periodic \cite[1.1.5]{LM}.
\end{itemize}
\end{ex}

\paragraph{\it Equivariant cohomology theories.}
We fix a base field $k$ of characteristic $0$. Let $G$ be a split semisimple (not necessarily simply-connected) linear algebraic group over $k$. Given an algebraic oriented cohomology theory $\hh$ in the sense of \cite{LM} one constructs the associated $G$-equivariant oriented cohomology theory as follows \cite[\S5]{HML}:

Consider a system of $G$-representations $V_i$ and its open subsets $U_i\subseteq V_i$ such that
\begin{itemize}
\item $G$ acts freely on $U_i$ and the quotient $U_i/G$ exists as a scheme over $k$,
\item $V_{i+1}=V_i\oplus W_i$ for some representation $W_i$,
\item $U_i\subseteq U_i\oplus W_i\subseteq U_{i+1}$, and $U_i\oplus W_i\to U_{i+1}$ is an open inclusion, and
\item  $\mathrm{codim}\,(V_i\setminus U_i)$ strictly increases.
\end{itemize}
Such a system is called a good system of representations of $G$. 

Let $X$ be a amooth $G$-variety. Following~\cite[\S3 and \S5]{HML} the inverse limit induced by pull-backs 
\[
\varprojlim_i \hh(X\times^G U_i),\quad X\times^G U_i=(X\times_k U_i)/G,
\] 
does not depend on the choice of the system $(V_i,U_i)$ and, hence, defines the $G$-equivariant oriented cohomology $\hh_G(X)$.

The key property of $\hh$ and, hence, of $\hh_G$ is that it has characteristic classes which satisfy the Quillen formula for the tensor product of line bundles \cite[Lemma~1.1.3]{LM}. Namely,
\[
c_1^{\hh}(\Ll_1\otimes \Ll_2)=F(c_1^{\hh}(\Ll_1),c_1^{\hh}(\Ll_2)),
\]
where $c_1$ is the first characteristic class in the theory $\hh$, $\Ll_i$ is a line bundle over some variety $X$ and $F$ is a f.g.l. Hence, given $\hh$ one associates the f.g.l. $F$  over the coefficient ring $R=\hh(pt)$ (here $pt=\mathrm{Spec}\, k$). Conversely, given $F$ over $R$ one defines
\[
\hh(-):=\Omega(-)\otimes_{\LL}R,
\] 
where $\Omega$ is the universal oriented theory (algebraic cobordism of Levine-Morel) over the $\Omega(pt)=\LL$. Observe that all such theories satisfy the localization sequence
\[
\hh(Z) \to \hh(X) \to \hh(U) \to 0,
\]
where $U$ is open in $X$ and $Z=X\setminus U$ has the reduced subscheme structure. 

\begin{ex}
The Chow theory $\hh=\CH$ corresponds to the additive f.g.l. over $R=\Z$ \cite[Ex.1.1.4]{LM} and $\CH_G$ is the $G$-equivariant Chow theory in the sense of Totaro \cite{To,EG}.

Following \cite[Ex.1.1.5]{LM} consider the graded Grothendieck $K^0$ 
\[
K^0(X)[\be^{\pm 1}]=K^0(X)\otimes_\Z \Z[\be^{\pm1}],
\] 
where $\be$ is a formal variable (here we assign $\deg \be=-1$ and $\deg x=0$ for all $x\in K^0(X)$). Then $K^0(X)[\be^{\pm 1}]$ is an oriented theory in the sense of \cite{LM}  which corresponds to the multiplicative periodic f.g.l. over $R=\Z[\be^{\pm 1}]$. If we set $\be=1$, then the $\Z$-graded version of $K^0$ is the direct sum of copies of the usual $K^0(X)$ with each copy assigned the respective degree. Then $K_G^0(X)$ is the $\Z$-graded version of the equivariant $K$-theory obtained using the Borel construction.
\end{ex}

\begin{ex}
If $F$ is the f.g.l. of an elliptic curve, then it was shown in \cite[Rem.4.3]{ZZ} that the associated oriented cohomology theory $\hh_G$  is a stalk at the origin of the equivariant elliptic cohomology constructed by Ginzburg-Kapranov-Vasserot \cite{GKV}.
\end{ex}

\begin{ex}
The $G$-equivariant algebraic cobordism $\Omega_G$ constructed in \cite{HML} corresponds to a universal f.g.l. $F_U$ over the Lazard ring $\LL$.
\end{ex}

\paragraph{\it Equivariant coefficient rings.}
Let $T$ be a split maximal torus of $G$ and let $T^*$ be its character group. 

Suppose $F$ is not a polynomial. Following \cite[\S2]{CPZ} consider a completion $R[[x_\la]]_{\la\in T^*}$ of the polynomial ring $R[x_\la]_{\la\in T^*}$ in variables $x_\la$ with respect to the kernel of the augmentation map $x_\la\mapsto 0$. Set
\[
S=R[[x_\la]]_{\la\in T^*}/J
\] 
where $J$ is the closure of the ideal of relations $x_0=0$ and $x_{\la+\mu}=x_\la+_F x_\mu$ for all $\la,\mu\in T^*$. The ring $S$ is called the formal group algebra associated to $G$ and $F$. According to \cite[\S3]{CZZ2}, the $R$-algebra $S$ is the completion of the $T$-equivariant coefficient ring $\hh_T(pt)$ along the kernel of the forgetful map $\hh_T(pt)\to R$, i.e.
\[
S \simeq \hh_T(pt)^\wedge.
\]
Under this identification $x_\la$ is the first characteristic class $c_1^{\hh}(\Ll_\la)$ of the associated $T$-equivariant line bundle $\Ll_\la$ over $pt$. For simplicity, we will always deal with the so called Chern-complete theories (see \cite[Definition~2.2]{CZZ2} and \cite[Remark~2.3]{CZZ2}), hence, assuming that $\hh_T(pt)^\wedge=\hh_T(pt)$.

Suppose $F$ is a polynomial, which is exactly the case for the additive and multiplicative periodic f.g.l. Then we set $S=R[x_\la]_{\la\in T^*}/J$. It is well-known that $S\simeq \hh_T(pt)$.

\begin{ex}\label{ex:chr} 
For the additive f.g.l., i.e., for the Chow theory, if $G=PGL_2$, then $S=\Z[\al]$ where $\al$ is a simple root. If $G=SL_2$, then $S=\Z[\om]$ where $\om$ is a fundamental weight ($\al=2\om$).
\end{ex}

\begin{ex}\label{ex:kth}
For the multiplicative periodic f.g.l. (see e.g. \cite[2.20]{CPZ}), i.e. for $K^0(-)[\be^{\pm 1}]$, the ring $S$ is the group ring $R[T^*]$, $R=\Z[\be^{\pm 1}]$ (here we identify $x_\la$ with $\be^{-1}(1-e^{-\la})$). For instance, if $G=SL_2$, then $S=R[t^{\pm 1}]$, where $t=e^{\om}$.
\end{ex}

We fix a Borel subgroup $B$ of $G$ containing $T$ and, hence, the root system $\Sigma$ of $G$, the decomposition $\Sigma=\Sigma^+\amalg \Sigma^-$ into positive and negative roots and the Weyl group $W$. We recall computations (which is essentially an extension of the arguments by Kostant-Kumar \cite{KK86, KK90}) of the equivariant oriented cohomology ring $\hh_T(G/B)$ of the complete flag variety. Our main source is \cite{CZZ2}. All the details and proofs can be either found in \cite{CZZ2} or in the related works \cite{HMSZ} and \cite{CZZ}.

For all these computations to work we need to assume that $S$ is $\Sigma$-regular of \cite[Definition~4.4]{CZZ}. In particular, this holds if $2$ is a non-zero divisor in $R$ or if the root datum of $G$ does not contain an irreducible component of type ${C}$.

\paragraph{\it The complete flag variety case.} 
The Weyl group $W$ acts on the group of characters $T^*$ and, hence, on the equivariant coefficient ring $S=\hh_T(pt)$. Following \cite[\S5]{CZZ2} we define the twisted group algebra $S_W$ as the free left $S$-module 
\[
S_W=S\otimes_R R[W].
\] 
So each element of $S_W$ can be written as a $S$-linear combination $\sum_{w\in W} q_w\de_w$, where $\{\de_w\}_{w\in W}$ is the standard basis of the group ring, $q_w\in S$ are coefficients and the multiplication on $S_W$ satisfies the twisted commuting relation
\[
w(q)\de_w =\de_w q,\quad  q\in S.
\]

Consider the complete flag variety $G/B$. The $S$-linear dual $S_W^\star$ can be identified with the $T$-equivariant oriented cohomology of the $T$-fixed point set of $G/B$ \cite[\S6]{CZZ2}, i.e., 
\[
S_W^\star=\Hom_S (S_W,S)= \Hom(W,S) \simeq \hh_T((G/B)^T)= \oplus_{w\in W} \hh_T(pt_w).
\]
Observe that $S_W^\star$ is a commutative ring where the product is given by the usual pointwise multiplication of functions $W\to S$.

Let $Q=S[\tfrac{1}{x_\al},\al\in \Sigma]$ denote the localization of $S$ at all $x_\al$'s (i.e. at all characteristic classes $c_1^{\hh}(\Ll_\al)$) and let $Q_W=Q \otimes_R R[W]$ be the respective twisted group algebra. Following \cite[\S6]{HMSZ} we define the formal affine Demazure algebra $\DF$ to be the subalgebra of $Q_W$ generated by elements of $S\subset Q_W$ and by the so-called push-pull (or Demazure) elements 
\[
Y_i=\tfrac{1}{x_{-\al_i}}+\tfrac{1}{x_{\al_i}}\de_{s_i},\text{ for all simple roots }\al_i.
\]
The constant part of $\DF$, i.e. the subalgebra generated only by $Y_i$'s over $R$, will be denoted by $D_F$.

\begin{ex} (cf. \cite[Prop.7.1]{HMSZ}) 
For the Chow theory, $D_F$ (resp. $\DF$) is the nil-Coxeter (resp. affine nil-Hecke) algebra. For the Grothendieck $K^0$, $D_F$ (resp. $\DF$) is the (affine) $0$-Hecke algebra.
\end{ex}

According to \cite[Thm.8.2]{CZZ2} there is a ring isomorphism $\hh_T(G/B)\simeq \DFd$ between the $T$-equivariant oriented cohomology and the $S$-dual of $\DF$. The natural inclusion $\eta\colon S_W \to \DF$ induces the inclusion $\eta^*\colon \DFd \to S_W^\star$ which is the restriction to the $T$-fixed point locus $\hh_T(G/B) \hookrightarrow \hh_T((G/B)^T)$.

\paragraph{\it The Borel description.}
The algebra $\DF$ acts on its dual $\DFd$ by means of the Hecke action `$\bullet$', defined by ($Q_W$ acts on its dual $Q_W^*=\Hom_{Q}(Q_W,Q)$)
\[
q\de_w \bullet pf_v = p vw^{-1}(q)f_{vw^{-1}},\quad p,q\in Q,
\]
where $f_v$ is the standard dual basis (i.e., dual to $\de_w$) for the $Q$-module $Q_W^*$. Under this action the elements $Y_i$ of $\DF$ correspond to the classical push-pull operators 
\[
Y_i\bullet - \colon \hh_T(G/B)\to \hh_T(G/P_i) \to \hh_T(G/B),
\] 
where $P_i$ is the minimal parabolic subgroup corresponding to the simple root $\al_i$.

The equivariant characteristic map \cite[Def.~6.8]{CZZ1} $c\colon Q \to Q^*_W$ defined by $q\mapsto q\bullet \mathbf{1}$ restricts to $c\colon S \to \DFd=\hh_T(G/B)$. By definition, $c$ is a $W$-equivariant $R$-algebra homomorphism given by $c(q) = \sum_{w\in W}w(q)f_w$. Consider the map 
\[
\rho\colon S\otimes_{S^{W}} S \to \hh_T(G/B), \quad (s,s')\mapsto s\cdot c(s').
\]
Then by \cite[Thm.10.2]{CZZ2} the map $\rho$ is an isomorphism (called the Borel isomorphism) if and only if the usual characteristic map $c\colon S \to \hh(G/B)$ is surjective.

\begin{rem}
Observe that the Borel isomorphism holds if the ring of coefficients $R$ contains $\Q$ or if $\hh=K_0$ and $G$ is simply-connected \cite{St}. For Chow groups with $\Z$-coefficients it fails in general (e.g. for any simple group of type $B_n$, $n\ge 2$).
\end{rem}

\paragraph{\it The parabolic case.}
Let $P$ be a parabolic subgroup of $G$ and let $G/P$ be the respective projective homogeneous variety. Let $W_P$ denote the Weyl group of the Levi-part of $P$. The above constructions can be extended to describe the $T$-equivariant oriented cohomology $\hh_T(G/P)$ as follows (see \cite[pp.126-128]{CZZ2}).

We first define $\DP$ to be the image of $\DF$ under the composite $\DF\to Q_W \to Q_{W/W_P}$, where $Q_{W/W_P}$ is a free $Q$-module spanned by  elements $\de_{\bar w}$ corresponding to cosets $\bar w\in W/W_P$. Observe that $\DP$ is not a ring but only a free left $S$-module. We set $\DPd=\Hom_S(\DP,S)$. The product on $Q_W^*$ descends to the product on $\DPd$  (see \cite[\S11]{CZZ2}) and turns it into an algebra. The canonical map $\DPd \to \DFd$ induced by the projection $Q_{W} \to Q_{W/W_P}$ is injective, moreover, its image can be identified with the $W_P$-invariant subring via the Hecke action `$\bullet$'.

Finally, there are ring isomorphisms (see \cite[Thm.~8.11 and Thm.~9.1]{CZZ2})
\[
\hh_T(G/P)\simeq (\DFd)^{W_P}\simeq \DPd.
\]
The restriction to the $T$-fixed point locus $\hh_T(G/P) \to \hh_T((G/P)^T)$ is then induced by $S_{W/W_P} \to \DP$. The Borel map $\rho$ restricts to $W_P$-invariants as well and turns into \cite[Lemma~11.1]{CZZ2}
\[
\rho_P\colon S\otimes_{S^{W}} S^{W_P} \to \hh_T(G/P).
\]
Therefore, if $R$ contains $\Q$ we can identify $\hh_T(G/P)$ with $S\otimes_{S^{W}} S^{W_P}$.


\section{Equivariant oriented cohomology of the products}\label{sec:orcohprod}

We now fix two parabolic subgroups $P$ and  $Q$  of $G$ and  denote by $W_P$ (resp. $W_Q$) the Weyl groups of the Levi part of $P$ (resp. $Q$). 

Our first goal is to describe the cohomology of the product $\hh_G(G/Q\times G/P)$ as a direct sum spanned by classes (of desingularizations) of orbits parametrized by double cosets $W_Q\backslash W/W_P$ (see Lemma~\ref{lem:celldec}). Note that for Chow theory (even for Chow motives) this was done in \cite[Thm.~16]{CM}; for an arbitrary $\hh_G$ and $P=Q=B$ this is \cite[Example~3.6]{NPSZ}. We also show that this direct sum decomposition is compatible with the forgetful map (Corollary~\ref{cor:forget}) and with the geometric action of $W$ (Corollary~\ref{cor:Wact}). Finally, using the K\"unneth isomorphism of \cite{NPSZ} and the algebraic $W$-action of \cite{LZZ} we describe the $W$-invariants of the cohomology of the product as $W_Q$-invariants of the cohomology of $G/P$ (Proposition~\ref{prop:winvis}).

\paragraph{\it Orbit stratification.} 
We follow \cite[\S2-5]{CM}. Consider diagonal action of $G$ on the product $X=G/Q\times G/P$. The $G$-orbits are parametrized by double cosets $W_{Q}\backslash W/W_{P}$: the orbit $\Os_u$ corresponding to a minimal double coset representative $u$ is the orbit of the pair $(Q,uPu^{-1})$. The closure $\overline{\Os}_w$ contains $\Os_v$ for another minimal double coset representative $v$ if and only if $v\le w$ in the Bruhat order on $W$. For each orbit $\Os_u$ there is a $G$-equivariant affine fibration of rank $l(u)$ (see~\cite[Prop.11]{CM})
\[
\phi_u\colon \Os_u \simeq G/(Q \cap uPu^{-1}) \to G/P_u,
\]
where $P_u:=R_{unip}(Q)\cdot (Q\cap uPu^{-1})$ is a standard parabolic subgroup. So $X$ has the $G$-equivariant stratification of  \cite[Prop.13]{CM} by closures of the orbits $\overline{\Os}_u$ with the base consisting of flag varieties $\{G/P_u\}_{u\in {}^QW^P}$, where ${}^QW^P$ denotes the set of minimal double coset representatives.

\paragraph{\it Relatively cellular filtrations}
Given $u\in {}^Q W^P$ let $u=s_{i_1}s_{i_2}\ldots s_{i_l}$ be its reduced expression (written as a product of simple reflections) and let $P_i$ denote the minimal parabolic subgroup corresponding to $s_i$. Then the image of the projection on the first and the last factors (c.f. \cite[(8)]{NPSZ})
\[
\widehat{\Os}_u=G/B\times_{G/P_{i_1}} G/B \times_{G/P_{i_2}} \times \ldots \times G/B \to G/B\times G/B  \to G/Q \times G/P
\]
coincides with the closure $\overline{\Os}_u$. Therefore it gives a $G$-equivariant resolution of singularities of $\overline{\Os}_u$. Combining it with the orbit stratification we obtain that $X$ satisfies the following slight modification of \cite[Definition~3.2]{NPSZ} (cf. \cite[\S66]{EKM}):
\begin{itemize}
\item[(1)] there is a filtration by proper closed $G$-subvarieties $X_u=\overline{\Os}_u$, $u\in {}^Q W^P$ (inclusions correspond to the Bruhat order `$<$')
\item[(2)] each $\Os_u=X_u \setminus \cup_{v<u} X_v$ is a $G$-equivariant affine fibration over $Y_u=G/P_u$ of rank $l(u)$, and
\item[(3)] each $X_u$ admits a $G$-equivariant resolution of singularities $\widehat X_u =\widehat{\Os}_u$; we set $g_u\colon \widehat X_u \to X_u \hookrightarrow X$ and hence $[\widehat X_u]:=g_{u*}(1_{\widehat X_u})\in \hh_G(X)$ (the image of the fundamental class of $\hat X_u$ under the induced push-forward map).
\end{itemize}
Similar to \cite[\S66]{EKM} we call such $X$ a relatively $G$-equivariant cellular space over the base $\{Y_u\}$.

\begin{rem}
A difference from \cite[3.2]{NPSZ} is that in (2) instead of a fixed base $S$ we are allowed to have several base varieties $Y_u$. Observe also that for $\hh=\CH$ conditions (1) and (2) are precisely those of \cite[Prop.13]{CM}.
\end{rem}

\paragraph{\it Cohomology decomposition.}
We follow the arguments of  \cite[\S66]{EKM} for the variety $X=G/Q\times G/P$, $X_u=\overline{\Os}_u$, $U_u=\Os_u$, $Y_u=G/P_u$, and $Y_e=\Os_e$ (the minimal closed orbit). Let $\Gamma_u \subset U_u \times Y_u$ be the graph of the fibration $\phi_u$, let $\overline{\Gamma}_u$ denote its closure in $X_u \times Y_u$. Consider the  composite
\[
a_u\colon \CH(Y_u) \stackrel{p_Y^*}\to \CH(X_u \times Y_u) \stackrel{-\cap \overline{\Gamma}_u}\longrightarrow \CH(X_u\times Y_u) \stackrel{p_{X*}}\to \CH(X_u) \to \CH(X),
\]
where the maps $p_Y$ and $p_X$ are projections, the second map is the multiplication by the class of $\overline{\Gamma}_u$ and the last map is the push-forward induced by the closed embedding $X_u\hookrightarrow X$. By definition $a_u(1_{Y_u})=[X_u]$. Then \cite[Thm.~66.2]{EKM} says that the sum $\oplus_u a_u$ gives the direct sum decomposition
\[
\bigoplus_{u\in {}^Q W^P} \CH(Y_u)[X_u] \stackrel{\simeq}\to \CH(X).
\]
(Note also that this decomposition is exactly the one of \cite[Thm.16]{CM} obtained by taking realizations of the respective Chow motives.)

To extend it to an arbitrary $\hh$ we replace $X_u$ and $\overline{\Gamma}_u$ by the respective desingularizations $\widehat{X}_u$ and $\widehat{\Gamma}_u$ outside of $U_u$ (cf. the proof of \cite[Thm.~4.4]{NZ}) so that there is a commutative diagram
\[
\xymatrix{
\hh(\widehat{X}_u \times Y_u) \ar[r]^{-\cap \widehat{\Gamma}_u}  &     \hh(\widehat{X}_u\times Y_u) \ar[d] \\
\hh(Y_u)\ar[r]^{\phi_u^*} \ar[u]^{p_{Y}^*} & \hh(U_u)
}
\]
Consider the respective composite
\[
\hat a_u\colon \hh(Y_u) \stackrel{p_Y^*}\to \hh(\widehat{X}_u \times Y_u) \stackrel{-\cap \widehat{\Gamma}_u}\longrightarrow \hh(\widehat{X}_u\times Y_u) \stackrel{p_{X*}}\to \hh(X).
\]
We have $\hat{a}_u(1_{Y_u})=[\widehat{X}_u]$ and by the same arguments as in  \cite[Example~2.3]{NPSZ} we conclude that $\oplus_u \hat a_u$ is an isomorphism, hence, it gives a direct sum decomposition of $\hh(pt)$-modules
\[
\bigoplus_{u\in  {}^Q W^P} \hh(Y_u)[\widehat{X}_u] \stackrel{\simeq}\to \hh(X).
\]

Since all the maps and desingularizations are $G$-equivariant, repeating the arguments of \cite[Example~3.6]{NPSZ} we obtain the $G$-equivariant version of this decomposition:

\begin{lem}\label{lem:celldec} There is a direct sum decomposition of $\hh_G(pt)$-modules
\[
\bigoplus_{u\in  {}^Q W^P} \hh_G(G/P_u)[\widehat{\Os}_u] \stackrel{\simeq}\to \hh_G(G/Q\times G/P).
\]
\end{lem}

In other words, any element $x\in \hh_G(G/Q\times G/P)$ can be written uniquely as 
\[
x=\sum_{u\in {}^Q W^P} x_u [\widehat{\Os}_u] ,\quad \text{ where }x_u\in \hh_G(G/P_u).
\]

\begin{ex}
If $P=B$ (resp. $Q=B$), then $P_u=B$ for all $u$ and the orbit stratification of $G/Q\times G/P$ coincides with the relative cellular stratification induced by projections $G/Q\times G/P$ onto $G/P$ (resp. $G/Q$). In other words, the base of this stratification is $Y_u=G/B$ and the respective cohomology $\hh_G(G/Q\times G/P)$ is a free $S$-module:
\[
\bigoplus_{u\in {}^QW^P} S [\widehat{\Os}_u]\stackrel{\simeq}\to \hh_G(G/Q \times G/P). 
\]
\end{ex}

\paragraph{\it Restricting to $T$-varieties} 
Given a $G$-variety $X$ let $\hh_{G|T}(X)$ denote the image of the forgetful map $\hh_G(X)\to \hh_T(X)$. Observe that $\hh_{G|T}(X)$ is a subring of the $T$-equivariant cohomology $\hh_T(X)$.

For $X=G/Q\times G/P$ the forgetful map $\hh_G(X)\to \hh_T(X)$ is induced by restricting the coefficients $\hh_G(G/P_u) \to \hh_T(G/P_u)$ in the above decomposition. Hence, we obtain:

\begin{cor}\label{cor:forget} 
There is a direct sum decomposition of $\hh_{G|T}(pt)$-modules
\[
 \bigoplus_{u\in  {}^Q W^P} \hh_{G|T}(G/P_u)[\widehat{\Os}_u] \stackrel{\simeq}\to \hh_{G|T}(G/Q\times G/P).
\]
\end{cor}

\begin{rem} 
Observe that the forgetful maps $\hh_G(G/P_u) \to \hh_T(G/P_u)$ and hence $\hh_G(X)\to \hh_T(X)$ are not necessarily injective (see the discussion in \cite{CNZ}).

However, if either $P=B$ or $Q=B$ then the forgetful map \[S=\hh_G(G/P_u) \to \hh_T(G/P_u)=\hh_T(G/B)\] is injective, hence, so is $\hh_G(X)\to \hh_T(X)$. So we obtain
\[
\hh_G(G/Q\times G/B)=\hh_{G|T}(G/Q\times G/B)\text{ and } \hh_G(G/B\times G/P)=\hh_{G|T}(G/B\times G/P).
\]

Finally, there are examples (see \cite{CNZ}) of $G$, $P\neq B$ and $Q\neq B$ such that maps $\hh_G(G/P_u) \to \hh_T(G/P_u)$ are injective for all $u\in\! {}^Q W^P$ and, therefore, we have  $\hh_G(G/Q\times G/P)=\hh_{G|T}(G/Q\times G/P)$.
\end{rem}

We now study the $W$-action on the cohomology of the product.

\paragraph{\it Geometric $W$-action.}
For an arbitrary (left) $G$-variety $X$ there is a natural (left) action of $W$ on $\hh_T(X)$. It can be realized by pull-backs induced by a right action of $W$ on each step of the Borel construction 
\[
V\times^T X=V\times X/(v,x)\sim (vt,t^{-1}x),\quad t\in T
\] 
given by 
\[
(v,x)T\cdot\sigma T=(v\sigma,\sigma^{-1}x)T, \quad \sigma \in N_G(T)
\] 
where $V$ is a contractible space taken to have a right $G$-action. 

By definition the forgetful map $\hh_G(X) \to \hh_T(X)$ factors through the invariants $\hh_T(X)^W$. Hence, $\hh_{G|T}(X)$ is a subring of $\hh_T(X)^W$.

Since the orbit stratification of $X=G/Q\times G/P$ is $W$-equivariant we obtain:

\begin{cor}\label{cor:Wact} 
There is a direct sum decomposition of $S^W\!$-modules
\[
\bigoplus_{u\in {}^Q W^P} \hh_{T}(G/P_u)^W[\widehat{\Os}_u]   \stackrel{\simeq}\to \hh_{T}(G/Q\times G/P)^W. 
\]
\end{cor}

Observe that since $\hh_{G|T}(pt)$ is of finite index in $\hh_T(pt)^W$, $\hh_{G|T}(G/Q\times G/P)$ is a subgroup of finite index in $\hh_{T}(G/Q\times G/P)^W$.

\paragraph{\it Algebraic $W$-action.}
Following \cite[\S3]{LZZ} we introduce another $Q_W$ action on $Q_W^*$ by
\[
(p\de_w)\odot (qf_v)=pw(q)f_{wv}, ~p,q\in Q, w,v\in W,
\]
where $f_v$ is the standard basis (i.e., dual to $\de_w$) for $Q$-module $Q_W^*$. This action restricts to an action of $\DF$ on $\DFd$. Moreover, the respective action by $\de_w\in \DF$ gives the action by $w\in W$ on the cohomology $\DFd=\hh_T(G/B)$, i.e., the geometric $W$-action on $\hh_T(G/B)$ is given by the $\odot$-action.

\begin{rem} 
The $\odot$-action appeared in the context of Chow groups and $K$-theory in the works by Tymoczko, Brion, Peterson, Knutson and others.
\end{rem}

So we have two different actions on $\DFd$: the $\odot$-action (the left action) and the Hecke $\bullet$-action (the right action). There are two commutative diagrams \cite[Lemma~3.7]{LZZ}
\[
\xymatrix{S\otimes_{S^W} S \ar[r]^-{\rho}\ar[d]_{(z\cdot\_)\otimes 1} & \DFd \ar[d]^{z\odot\_} && S\otimes_{S^W}S\ar[r]^-{\rho}\ar[d]_{1\otimes (z\cdot\_)} & \DFd \ar[d]^{z\bullet\_}\\
S\otimes_{S^W}S \ar[r]^-{\rho} & \DFd && S\otimes_{S^W}S \ar[r]^-{\rho} & \DFd}
\]
where $\rho$ is the Borel map and $z\in S_W$ acts on $S$ by the usual left multiplication.

The $\odot$-action restrict to the parabolic case \cite[Lemma~3.11]{LZZ}, i.e., $\DF$ (and hence $W$) acts on $\DPd$ via
\[
(p\de_w)\odot (qf_{\overline{v}})=pw(q)f_{\overline{wv}}, ~p,q\in Q, w,v\in W,
\]
and $\overline{v}$ denotes the respective coset in $W/W_P$.

We know that $\hh_T(G/P)\simeq \DPd=(\DFd)^{W_P}\simeq  \hh_T(G/B)^{W_P}$ (invariants under the right $\bullet$-action of $W_P$). We claim that
\begin{lem}\label{lem:Winvc}
There is a ring isomorphism 
\[
\Hom_{\DF}(\DQd,\DPd)\simeq {}^{W_Q} \DPd={}^{W_Q}(\DFd)^{W_P}
\]
(invariants under the left $\odot$-action of $W_Q$).
\end{lem}

\begin{proof}
By \cite[Theorem~3.12]{LZZ} any $\DF$-equivariant (w.r.t. $\odot$-action) morphism $\phi$  is uniquely determined by its value $\phi([pt_Q])\in \DPd$ at the class $[pt_Q]=x_Q f_{\bar e}$ of the $T$-fixed point, where $x_Q=\prod_{\al\in \Sigma^-}x_\al/\prod_{\al\in \Sigma_Q^{-}}x_\al$. Since $W_Q$ fixes $[pt_Q]$, it also fixes $\phi([pt_Q])$. Hence, we obtain an inclusion $\Hom_{\DF}(\DQd,\DPd)\hookrightarrow {}^{W_Q} \DPd$, $\phi\mapsto \phi([pt_Q])$. 

Observe that there is an isomorphism $\Hom_{Q_W}(Q_{W/W_Q}^*,Q_{W/W_P}^*)={}^{W_Q} Q_{W/W_P}^*$ for localizations. Hence, given an element $x\in {}^{W_Q}\DPd\hookrightarrow {}^{W_Q} Q_{W/W_P}^*$ there is the corresponding $\psi\in \Hom_{Q_W}(Q_{W/W_Q}^*,Q_{W/W_P}^*)$, $\psi(f_{\bar e})=x$. Define $\phi(z):=\psi(\tfrac{1}{x_Q}\de_e\bullet z)$, $z\in \DQd$. Then $\phi$ is $\DF$-equivariant (as the $\odot$-action  and the $\bullet$-action commute by \cite[Lemma~3.4.(i)]{LZZ}) and $\phi([pt_Q])=x$.
\end{proof}

\paragraph{\it K\"unneth isomorphism.}
Since both $G/Q$ and $G/P$ are $T$-equivariant cellular spaces over $pt$, by \cite[\S3]{NPSZ} there is a K\"unneth isomorphism
\begin{equation}\label{eq:kunn}
\hh_T(G/Q\times G/P) \simeq \Hom_{S} (\hh_T(G/Q),\hh_T(G/P)),
\end{equation}
which sends $a\in \hh_T(G/Q\times G/P)$ to $a_*(x)=p_{Q*}(a\cdot p_P^*(x))$, where $p_P$, $p_Q$ are the respective projections $G/Q\times G/P\to G/P,G/Q$.
Since the projections are $W$-equivariant, we obtain 
\[
a_*(wx)=p_{Q_*}(a\cdot p_P^*(wx))=p_{Q_*}(ww^{-1}a\cdot w p_P^*(x))=w ((w^{-1}a)_*(x)).
\]
Hence, the K\"unneth isomorphism restricts to 
\[
\hh_T(G/Q\times G/P)^W \simeq \Hom_{W\text{-equiv}} (\hh_T(G/Q),\hh_T(G/P)).
\]
Since the left hand side can be identified with $\Hom_{\DF}(\DQd,\DPd)$, by Lemma~\ref{lem:Winvc}  we obtain
\begin{prop}\label{prop:winvis}
There is a ring isomorphism
\[
\hh_T(G/Q\times G/P)^W \simeq \hh_T(G/P)^{W_Q}.
\]
\end{prop}

\begin{ex}
In the case of Borel isomorphism (e.g. if $R$ contains $\Q$), we have $\hh_T(G/P)^W=S^{W_P}$, $\hh_T(G/Q)^W=S^{W_Q}$ and $\hh_T(G/Q\times G/P)^W\simeq S^{W_Q}\otimes_{S^W} S^{W_P}$.
\end{ex}


\section{Double moment graphs}\label{sec:doublemom}

Consider a root system $\Sigma$ with a set of simple roots $\Theta=\{\al_1,\ldots,\al_n\}$ and a set of positive roots $\Sigma^+$. Let $W$ be the Coxeter group generated by the simple reflections $s_{\al_i}$. Given two subsets $\Theta_P$, $\Theta_Q \subset \Theta$ we denote by $\Sigma_P$, $\Sigma_Q$ the respective root subsystems and by $W_P$ and $W_Q$ the respective Coxeter subgroups. We denote by ${}^QW^P$ the set of minimal length double coset representatives of  $W_Q\backslash W/W_P$. If $\Theta_Q=\emptyset$ (resp. $\Theta_P=\emptyset$), then $W_Q$ (resp. $W_P$) is trivial and we denote by $W^P$ (resp. by ${}^QW$) the set of minimal left (resp. right) coset representatives of $W/W_P$ (resp. $W_Q\backslash W$).

The Bruhat order `$\leq$' on $W$ restricts to ${}^QW^P$ which we can hence consider as a poset itself. Given $u\in {}^QW^P$ consider the subset $\Theta_u=\Theta_Q\cap u(\Sigma^+_P)$ and the respective Coxeter subgroup $W_u$ generated by $\Theta_u$. It can be shown (though not needed) that $W_u=W_Q\cap u W_P u^{-1}$. Denote by $W_Q^u$ the set of minimal length coset representatives of $W_Q/W_u$. 

We will extensively use the following fact (Double Parabolic Decomposition, see, for example, \cite{C}): 
\begin{quote}
An element $y$ of a double coset $W_Q\backslash W/W_P$ can be uniquely written as $y=w u v$, where $u\in  {}^QW^P$, $w\in W_Q^u$ and $v\in W_P$. 
\end{quote}
In particular, any element of $W^P$ can be uniquely written as $wu$, where $u\in  {}^QW^P$ and $w\in W_Q^u$. Given $w\in W$ we denote by $\overline{w}$ (resp. $\widehat{w}$) the minimal length coset representative of $w$ in $W^P$ (resp. in ${}^Q W^P$).

We will also use the following definition of a moment graph that can be found by now in several papers, see, for example, \cite{BMP, F1}.

\begin{dfn} 
The data $\Gg=((V,\le), l\colon E\to \Sigma^+)$ is called a moment graph if
\begin{itemize}
\item[(MG1)] $V$ is a set of vertices together with a partial order `$\le$', i.e., we are given a poset $(V,\le)$.
\item[(MG2)] $E$ is a set of directed edges labelled by positive roots $\Sigma^+$ via the label function $l$, i.e., $E\subset V\times V$ with and edge $(x,y)\in E$ denoted $x\to y$ and labelled by $l(x \to y)\in \Sigma^+$.
\item[(MG3)] For any edge $x\to y\in E$, we have $x\le y$, $x\neq y$, i.e., direction of edges respects the partial order.
\end{itemize}

A moment graph $\Gg'=((V',\le), l\colon E'\to \Sigma^+)$ is called a moment subgraph of~$\Gg$ if $(V',\le)$ is a subposet of $(V,\le)$, $E'\subset E$ and $l'\colon E'\to \Sigma^+$ is the restriction of~$l$. Observe that (MG2) and (MG3) imply that the graph does not have multiple edges or self-loops: (MG2) disallows several edges between the same two vertices in the same direction, and (MG3) disallows pairs of edges between two vertices in the opposite directions and self-loops. Also, (MG3) implies that $\Gg$ has no directed cycles. Note also that the direction of edges in $\Gg$ is uniquely determined by the partial order `$\le$'.
\end{dfn}

\begin{ex} 
Consider the usual Bruhat poset $(W,\le)$. The data 
\[
V:=W,\; E:=\{w\to s_\al w\mid  w\le s_\al w,\, \al\in \Sigma^+\}\;\text{ and }\; l(w\to s_\al w):=\al
\]
define a moment graph. Observe that the transitive closure of $E$ gives the Bruhat order and $E$ contains all cover relations of the Bruhat order. 
\end{ex}

\begin{ex}\label{ex:parabolic}
Given a subset $\Theta_P$ of the set of simple roots $\Theta$, consider the Bruhat subposet $(W^P,\le)$.The data 
\[
V=W^P,\; E=\{w\to \overline{s_\al w}\mid w\le s_\al w,\; \overline{w}\neq \overline{s_\al w},\; \al\in \Sigma^+\}\;\text{ and }\;l(w\to \overline{s_\al w})=\al
\] 
define a moment graph called the parabolic moment graph (associated to $P$) and denoted by $\GP$. Again the transitive closure of $E$ gives the Bruhat order `$\le$' on $W^P$.

To show that $\GP$ has a well-defined label function (i.e. (MG2) holds) we follow \cite{St05}. Let $\Sigma_\R$ be a real vector space with an inner product $(\cdot,\cdot)$ that contains $\Sigma$. Choose a dominant $\theta\in \Sigma_\R$ such that $W_P$ is a stabilizer group of $\theta$. The poset $(W^P,\le)$ can be identified with the poset $(W\theta, \le)$, via $w\mapsto w\theta$ (see \cite[Prop.~1.1]{St05}), where the partial order on the orbit $W\theta$ is defined by taking the transitive closure  of the relations 
\[
\mu< s_\al(\mu)\text{ for all }\al\in \Sigma^+\text{ such that }(\mu,\al)>0.
\]
Now suppose $w\to \overline{s_\al w}$, $w\to \overline{s_{\be} w} \in E$ are such that $\overline{s_\al w}=\overline{s_\be w}$ for $\al,\be\in \Sigma^+$, $w\in W^P$. Then $s_\al(\mu)=s_\be(\mu)$, where $\mu=w\theta< s_\al(\mu)$, which implies that $\al=\be$. In other words, the label $\al$ is uniquely determined by the representative $\overline{s_\al w}$. 
\end{ex}

As the next example shows the definition of $\GP$ can not be directly extended to double cosets

\begin{ex} 
Consider the symmetric group $W=S_3$ generated by simple reflections $s_1,s_2$ corresponding to simple roots $\al_1,\al_2$. Let $W_Q=W_P=\langle s_1 \rangle$. So we have ${}^Q W^P=\{e, s_2\}$. Similar to the parabolic case consider the directed graph given by
\[
V={}^Q W^P\;\text{ and }\; E=\{u\to \widehat{s_\al u}\mid u\le s_\al u,\;\widehat{u}\neq\widehat{s_\al u},\;\; \al\in \Sigma^+\} 
\]
Since $\widehat{s_2e}=\widehat{s_{\al_1+\al_2}e}=s_2$, an edge $e\to s_2$ then has to be labelled by two different positive roots $\al_2$ and $\al_1+\al_2$.
\end{ex}

To avoid this problem, among all labels corresponding to $u\to \widehat{s_\al u}$ we simply choose one. This leads to the following key definition

\begin{dfn}\label{dfn:double} 
Given subsets $\Theta_P$, $\Theta_Q$ of the set of simple roots  $\Theta$, the moment graph on the poset $({}^QW^P,\le)$ with edges 
\[
{}^QE^P=\{u\to u' \mid \exists \be \in \Sigma^+ \, u' = \widehat{s_\be u},\;u \le s_\be u,\, u \ne u' \}
\] 
and labels obtained by choosing a unique $\al\in\Sigma^+$ for each edge $u\to  \widehat{s_\al u}$ is called the $(P,Q)$-double moment graph and denoted by $\QGP$. When the subsets $\Theta_P$ and $\Theta_Q$ are clear from the context we will simply say a double moment graph.

In other words, we have
\[
{}^QE^P=\{u\to u'=\widehat{s_{l(u\to u')} u}, \text{ where $l(u\to u')$ is the chosen label}\}.
\]
\end{dfn}

Though it depends on choices of the labels $l(u\to u')$ we will show that in most of the cases this choice is made only up to an action of $W_Q$, i.e., all multiple labels belong to the same $W_Q$-orbit and all labels in the $W_Q$-orbit can appear. To formalize the latter we introduce the following

\begin{dfn}
Let $\Gg=((V,\le),l\colon E\to \Sigma^+)$ be a moment graph. Suppose a subgroup $H$ of the Weyl group $W$ acts on the set $V$. We say that $\Gg$ is $H\!$-closed if 
\begin{enumerate}
\item[(MGE)] For all $w\in H$ and $x \to y\in E$  
\end{enumerate}
\begin{itemize}
\item if $w(l(x\to y))\in \Sigma^+$, then $w(x) \to  w(y)\in E$ and $l(w(x) \to  w(y))=w(l(x\to y))$, 
\item if $w(l(x\to y))\in \Sigma^-$, then $w(y) \to  w(x)\in E$ and $l(w(y) \to  w(x))=-w(l(x\to y))$.
\end{itemize}
\end{dfn}

Note that a $W\!$-closed moment graph is not a $W$-moment graph in the sense of \cite[Def.~4.2]{FL}.

\begin{lem}
The parabolic moment graph $\GP$ of Example~\ref{ex:parabolic} is $W\!$-closed, where $W$ acts on $W^P$ by $w(v)=\overline{wv}$, $v\in W^P$.
\end{lem}

\begin{proof} 
Suppose $y=\overline{s_\al x}$ with $x\le y$, $\overline{x}\neq\overline{y}$, and $\al\in \Sigma^+$, then $y=s_\al x z$, $z\in W_P$ and we have 
\[
w(y)=w(s_\al x z)=\overline{s_{w(\al)}w(x)}.
\]
Since $\overline{x}< \overline{y}$, we have $x(\theta)=\mu < s_\al(\mu)=y(\theta)$ in $W\theta$, i.e.,  $(\mu,\al)>0$. Then $(\mu,\al)=(w(\mu),w(\al))>0$ which implies that $w(x)(\theta)=w(\mu)< s_{w(\al)}w(\mu)=w(y)(\theta)$ if $w(\al)\in \Sigma^+$ and $w(\mu)> s_{w(\al)}w(\mu)$ if $w(\al)\in \Sigma^-$.
\end{proof}

We then introduce the notion of the $H$-closure of a subgraph.

\begin{dfn} 
Let $\Gg=((V,\le),l\colon E\to \Sigma^+)$ be an $H$-closed moment graph with $H\subset W$. Let $\Gg'=((V',\le),l\colon E'\to \Sigma^+)$ be a subgraph of $\Gg$. By the $H$-closure of $\Gg'$ we call the  subgraph of $\Gg$ denoted by $H(\Gg')$ with vertices $\{wx\mid w\in H,\; x\in V'\}$ and the following edges
\begin{itemize}
\item
$y \to z\in E$ such that there exist $x \in V'$, $v,w \in H$ with $y=w(x)$, $z=v(x)$, 
\item $w(x) \to  w(y)\in E$ labelled by $w(l(x \to y))$ for all $x\to y\in E'$ and $w \in H$ with $w(l(x \to y))\in \Sigma^+$,
\item $w(y) \to  w(x)\in E$ labelled by $-w(l(x \to y))$ for all $x\to y\in E'$ and $w \in H$ with $w(l(x \to y))\in \Sigma^-$. 
\end{itemize}
Roughly speaking, $H(\mathcal{G'})$ is the smallest $H\!$-closed subgraph in $\Gg$ containing $\Gg'$ and all edges of $\Gg$ inside the $H$-orbits of $V'$.
\end{dfn}

Consider a double moment graph  $\QGP$ of Definition~\ref{dfn:double}. Since each minimal double coset representative is also a minimal left coset representative, the double moment graph can be naturally viewed as a subgraph $\Gg'$ of $\GP$ with  
\[
V'=\bigcup_{u\to u'\in {}^Q E^P}\{u,\overline{s_{l(u\to u')} u}\}\text{ and }E'=\{u\to \overline{s_{l(u\to u')} u}\mid u\to u'\in {}^QE^P\}
\]
and the same labels $l(u\to u')$.

\begin{dfn}
The $W_Q$-closure of $\Gg'$ in $\GP$ is called the closure of $\QGP$ and is denoted by $\langle \QGP\rangle$. 
\end{dfn}

\begin{rem}
By definition, a double moment graph $\QGP$ and hence its closure  $\langle \QGP\rangle$ depends on a choice of labels. So for two double moment graphs ${}^Q\Gg_1^P$ and ${}^Q\Gg_2^P$ with the same sets of vertices and edges but with different choices of labels, it is possible that $\langle {}^Q\Gg_1^P\rangle \neq \langle {}^Q\Gg_2^P\rangle$ in $\GP$.
\end{rem}

To compute the closure $\langle \QGP\rangle$ we will use the following

\begin{lem}\label{lem:geomorb}
Each edge of $\GP$ is either of the form $wu \to vu$, where $w,v \in W_Q$, and $u\in {}^Q W^P$, or it can be obtained by applying an element of $W_Q$ to an edge of the form $u \to \overline{s_\al u}$, where $u\in {}^Q W^P$ and $\widehat{s_\al u} \ne u$.
\end{lem}

\begin{proof}
Let $\theta\in \Sigma_\R$ be a dominant vector with stabilizer $W_P$. Following \cite[\S1B]{St05} consider the subposet of $W\theta$
\[
(W\theta)_{Q}=\{\mu \in W\theta\mid (\mu,\al)\ge 0\text{ for all }\al\in \Sigma_{Q}^+\}
\]
It is isomorphic to the Bruhat subposet $({}^Q W^P,\le)$ via the evaluation map $w\mapsto w\theta$ \cite[Prop.1.5]{St05}.

Consider an edge $x\to y$ in $\GP$ labelled by $\al\in \Sigma^+$. Suppose $\widehat{x}\neq \widehat{y}$, then $\widehat{x} <\widehat{y}$ by \cite[Prop.2.5.1]{BB}. By parabolic decomposition, we have $x=wu$, where $w\in W_Q^u$ and $u\in {}^Q W^P$. Hence,  it corresponds to a vector $w(\mu) \in W\theta$, where $\mu=u(\theta)\in (W\theta)_Q$. Similarly, $y=\overline{s_\al x}$ corresponds to a vector $s_\al w(\mu)$ in $W\theta$. So we have $w(\mu) < s_\al w(\mu)$ that is $(w(\mu),\al)=(\mu,w^{-1}(\al))>0$.

Suppose $w^{-1}(\al)\in \Sigma^+$, then the edge $x\to y$ is obtained from the edge \[u=w^{-1}x\to \overline{w^{-1}y}=\overline{s_{w^{-1}(\al)} u}\] labelled by $w^{-1}(\al)$ by applying $w\in W_Q^u$ (here $\mu < s_{w^{-1}(\al)}\mu$ in $W\theta$).

If $w^{-1}(\al) \in \Sigma^-$, then $\mu > s_{w^{-1}(\al)}\mu$ which contradicts $\widehat{x}< \widehat{y}$.
\end{proof}

We are now ready to introduce another key notion:
\begin{dfn}\label{dfn:closure}
The double moment graph $\QGP$ is called closed if $\langle \QGP\rangle=\GP$.
\end{dfn}

As a consequence of Lemma~\ref{lem:geomorb} and its proof we obtain
\begin{cor}\label{lem:clos} 
The following are equivalent
\begin{enumerate}
\item\label{cond:transonedges}
For every $\mu \in (W\theta)_Q$ and for every two outgoing edges $\mu \to s_{\al}\mu$ and $\mu \to s_{\be}\mu$ such that $s_{\be}\mu \in W_Q s_{\al}\mu$, but $s_{\al}\mu, s_{\be}\mu \notin W_Q \mu$, there exists $w\in W_Q$ such that $w\mu=\mu$ and $\be = w(\al)$ (in other words, $w(\mu \to s_{\al}\mu) = \mu \to s_{\be}\mu$).
\item
For every $u\in {}^Q W^P$ and  $\al,\be\in \Sigma^+$ such that $u<\widehat{s_\al u} = \widehat{s_\be u}$,  there exists  $w\in W_Q$ such that $u=\overline{wu}$ and $s_\be=s_{w(\al)}$.
\item
A double moment graph $\QGP$ is closed.
\item
Every double moment graph with the same vertices and edges as $\QGP$ (but possibly different labels) is closed.
\end{enumerate}
\end{cor}

\begin{proof} 
The equivalence $(1) \Longleftrightarrow (2)$ follows as $\mu=u(\theta)$.\\
$(2)\Rightarrow (4)$: If (2) holds, then any edge $u\to \overline{s_\be u}$ in $\GP$ can be obtained by applying some $w\in W_Q$ to the edge $u\to \overline{s_\al u}\in E'$. Hence, by Lemma~\ref{lem:geomorb}, any edge of $\GP$ can be obtained by applying some $w\in W_Q$ to an edge from $E'$.\\ 
$(4) \Rightarrow (3)$ is obvious.\\
$(3) \Rightarrow (2)$: Suppose $\QGP$ is closed and let $u\to \overline{s_\al u}\in E'$ be the respective edge in $\Gg'$ then any edge $u\to \overline{s_\be u}$ in $\GP$ with $\widehat{s_\al u}=\widehat{s_\be u}$ can be obtained by applying some $w\in W_Q$ to  $u\to \overline{s_\al u}$.
\end{proof}

\begin{ex}\label{ex:pqnothing} 
Observe that if $\Theta_Q=\emptyset$ then the double moment graph ${}^Q \GP$ coincides with the parabolic graph $\GP$.

Suppose $\Theta_P=\emptyset$ so that $W_P$ is trivial. Then $\GP$ coincides with the moment graph $\Gg$ for $W$ and the double moment graph ${}^Q \GP$ corresponds to a subgraph $\Gg'$ of $\Gg$ obtained by choosing minimal right coset representatives $u\in {}^Q W$ as vertices.

If $u<\widehat{s_\al u} = \widehat{s_\be u},\; u\in {}^Q W,\; \al,\be\in \Phi^+$ where $\widehat v$ is the minimal coset representative of  $W_Qv$, then $\exists w\in W_Q$ such that $s_\al u = ws_\be u$. Hence, $s_\al s_\be=w\in W_Q$ which implies that either both $s_\al,s_\be\in W_Q$ or $w=1$. Since $u\neq \widehat{s_\al u}=\widehat{s_\be u}$, we conclude that $w=1$ and $s_\al=s_\be$. Lemma~\ref{lem:geomorb} then implies that the $W_Q$-closure of $\Gg'$ coincides with $\Gg$.

Therefore, if $\Theta_P=\emptyset$ or $\Theta_Q=\emptyset$, then any double moment graph $\QGP$ is closed.
\end{ex}

\begin{ex}\label{lem:peverything}
If $\Theta_P=\Theta$, then $W_P=W$, and $\theta = 0$. So $\GP$ contains only one vertex $0$, and no edges.

If $\Theta_Q=\Theta$, we have $W_Q=W$. Observe that there is only one $W_Q$-orbit in $W \theta$, and it is not possible to find $\mu \in W \theta$ and $\al \in \Sigma^+$ such that $s_{\al}\mu \notin W_Q \mu$.

So by Corollary~\ref{lem:clos},  if $\Theta_P=\Theta$ or $\Theta_Q=\Theta$, then any double moment graph $\QGP$ is closed.
\end{ex}


\section{Closed double moment graphs}\label{sec:closedgr}

Observe that our main result (Theorem~ \ref{thm:main}.(B)) says that global sections of the structure sheaf on the closed double moment graph $\QGP$ describe equivariant cohomology of the product $G/Q\times G/P$. In the present section we provide necessary and sufficient conditions for $\QGP$ to be closed for all root systems $\Sigma$ of classical Dynkin types and of type $G_2$.

We first look at simply-laced cases.

\paragraph{\bf Type $A$.}  
Let $\Sigma_\R$ be the subspace of $\R^{n+1}$, $n\ge 1$ for which the coordinates sum to zero. The Weyl group $W=S_{n+1}$ acts on $\R^{n+1}$ by permutations of coordinates. We fix a set of simple roots $\Theta=\{\al_1,\ldots,\al_n\}$.  

Consider the vector $\mu=u(\theta)\in W\theta$, where $\theta$ is dominant with the stabilizer $W_P$ and $u\in {}^Q W^P$. The subgroup $W_Q$ acts on a vector $\mu=(x_1,\ldots,x_{n+1})$ by permutations of coordinates in each of the tuples $\tau_0=(x_1,\ldots,x_{k_1})$, $\tau_1=(x_{k_1+1},\ldots,x_{k_2})$, $\ldots$, $\tau_r=(x_{k_r+1},\ldots,x_{n+1})$, where $\Theta \setminus \Theta_Q=\{\al_{k_1},\ldots,\al_{k_r}\}$.

Consider an edge in the double moment graph connecting $u$ and $\widehat{s_\al u}$ (here we assume $n\ge 2$). Suppose $u<\widehat{s_\al u}=\widehat{s_\be u}$  for some $\al,\be\in \Sigma^+$. Let $s_\al=(i'j')$, $i'<j'$. Since $s_\al(\mu)\neq \mu$, we have $x_{i'}\neq x_{j'}$. Since $\widehat{s_\al u}\neq u$, the coordinates $x_{i'}$ and $x_{j'}$ belong to different tuples, respectively. Similarly, if $s_\be=(i''j'')$, $i''<j''$, then the coordinates $x_{i''}\neq x_{j''}$ belong to different tuples. Since $\widehat{s_\al u}=\widehat{s_\be u}$ and the action by $W_Q$ preserves the numbers of equal coordinates in each tuple, the coordinates $x_{i'}$, $x_{i''}$ (resp. $x_{j'}$, $x_{j''}$) belong to the same tuple $\tau_i$ (resp. $\tau_j$, $i\neq j$), therefore, $x_{i'}=x_{i''}$ and $x_{j'}=x_{j''}$.

Now set $w=(i'i'')(j'j'')$ to be the product of two commuting transpositions. Then, $w\in W_Q$ stabilizes $\mu$ and $s_\be=s_{w(\al)}$ and by~\ref{lem:clos} we obtain

\begin{prop}\label{prop:typeAc}
In type $A$, any double moment graph $\QGP$ is closed.
\end{prop}

\paragraph{\bf Type $D$.} 
We realize a root system of type $D$ as the subset of vectors in $\R^n$, $n\ge 4$
\[
\Sigma=\{\pm(e_i+e_j),\pm(e_i-e_j)\mid i\neq j,\; i,j=1,\ldots,n\},
\]
where $e_1,\ldots,e_n$ are the standard basis vectors. We fix a set of simple roots 
\[
\Theta=\{\al_1=e_1-e_{2},\ldots,\al_{n-1}=e_{n-1}-e_n,\al_n=e_{n-1}+e_n\}.
\]
The Weyl group $W$ acts by permutations and even sign changes, i.e., reflections are given either by usual transpositions or by signed transpositions 
\[
(ij)\colon e_i\mapsto e_j,\; e_j\mapsto e_i\quad \widetilde{(ij)}\colon e_i\mapsto -e_j,\; e_j\mapsto -e_i.
\]

Consider the vector $\mu=u(\theta)\in W\theta$, where $\theta$ is dominant with stabilizer $W_P$ and $u\in {}^Q W^P$. Suppose $\widehat{s_\al u}=\widehat{s_\be u}> u$ for some $\al,\be\in \Sigma^+$. 

\underline{Suppose $\al_{n}\notin \Theta_Q$.} Then $W_Q$ is a subgroup of the symmetric group $S_n$ acting by permutations on 
$\{e_1,\ldots,e_n\}$. Hence, as in type $A$, it acts on a vector
$\mu=(x_1,\ldots,x_{n})$ by permutations of coordinates in each of the tuples 
$(x_1,\ldots,x_{k_1})$, $(x_{k_1+1},\ldots,x_{k_2})$, $\ldots$, $(x_{k_{r}+1},\ldots,x_{n})$, 
where $\Theta\setminus (\Theta_Q\cup\{\al_n\})=\{\al_{k_1},\ldots, \al_{k_r}\}$. We have the following cases for the reflection $s_\al$:
\begin{itemize}
\item[(1)] Let $s_\al=(i'j')$, $i'<j'$. As in type $A$, the coordinates $x_{i'}$ and $x_{j'}$ belong to different tuples 
and $x_{i'}\neq x_{j'}$.

\item[(2)] Let $s_\al=\widetilde{(i'j')}$, $i'<j'$. Since $s_\al(\mu)\neq \mu$, we have $x_{i'}\neq -x_{j'}$ and 
\begin{itemize}
\item[(a)] $x_{i'}$ and $x_{j'}$ belong to different tuples.
\item[(b)] $x_{i'}$ and $x_{j'}$ belong to the same tuple.
\end{itemize}
\end{itemize}
And we have similar cases (1), (2a) or (2b) for $s_\be=(i''j'')$ or $\widetilde{(i'',j'')}$. 
Since $\widehat{s_\al u}=\widehat{s_\be u}$ and the action by $W_Q$ preserves the numbers of equal coordinates in each tuple, 
$s_\be$ corresponds to the same case as $s_\al$ and it interchanges coordinates (changes sign) between the same tuples as $s_\al$, therefore, $x_{i'}=x_{i''}$ and $x_{j'}=x_{j''}$ (in the case (2b) we may permute $i''$ and $j''$). 
Taking $w=(i'i'')(j'j'')$ as before we obtain that $w\in W_Q$ stabilizes $\mu$ and $s_\be=s_{w(\al)}$.

\underline{Suppose  $\al_{n-1},\al_n\in\Theta_Q$.}
The subgroup $W_Q$ acts on the vector 
$\mu=(x_1,\ldots,x_{n})$ by permutations of coordinates in the first $r$ tuples 
$(x_1,\ldots,x_{k_1})$, $(x_{k_1+1},\ldots,x_{k_2})$, $\ldots$, $(x_{k_{r-1}+1},\ldots,x_{k_r})$, 
and by permutations and even sign changes on the last tuple
$\tau=(x_{k_{r}+1},\ldots,x_{n})$, where $\Theta\setminus \Theta_Q=\{\al_{k_1},\ldots, \al_{k_r}\}$.
We have the following cases for the reflection $s_\al$:
\begin{itemize}
\item[(1)] Let $s_\al=(i'j')$, $i'<j'$. As before, $x_{i'}\neq x_{j'}$ and  
$x_{i'}$, $x_{j'}$ belong to different tuples. We then have two subcases
\begin{itemize}
\item[(a)] $x_{j'}$ does not belong to $\tau$.
\item[(b)] $x_{j'}$ belongs to $\tau$.
\end{itemize}

\item[(2)] Let $s_\al=\widetilde{(i'j')}$, $i'<j'$. We have $x_{i'}\neq -x_{j'}$ and 
\begin{itemize}
\item[(a)] $x_{i'}$ and $x_{j'}$ belong to different tuples and $x_{j'}$ does not belong to $\tau$.
\item[(b)] $x_{i'}$ and $x_{j'}$ belong to different tuples and $x_{j'}$ belongs to $\tau$.
\item[(c)] $x_{i'}$ and $x_{j'}$ belong to the same tuple which is not $\tau$.
\end{itemize}
\end{itemize}
And we have similar cases for $s_\be=(i''j'')$ or $\widetilde{(i'',j'')}$. 
Observe that if both $s_\al$ and $s_\be$ correspond to the same case, 
then they interchange coordinates (signs) between the same tuples, hence, we can take the same $w$ as before.

Since $\widehat{s_\al u}=\widehat{s_\be u}$ and the action by $W_Q$ preserves the numbers of equal coordinates in the first $r$ tuples, 
$s_\al$ and $s_\be$ either both correspond to the case (1a), or they both correspond to the case (2a), or they both correspond to the remaining cases.
Hence, we reduce to the situation when $s_\al$, $s_\be$ correspond to different cases among (1b), (2b) or (2c).

Suppose $s_\al$ corresponds to (1b) and $s_\be$ corresponds to (2b), then we can take $w=(i'i'')\widetilde{(j'j'')}$ which is in $W_Q$, stabilizes $\mu$ and
$\be=w(\al)$.

Suppose $s_\al$ corresponds either to (1b) or (2b) and $s_\be$ corresponds to (2c). 
Without loss of generality we may assume $x_{j''}=0$, which implies 
$x_{i'}=x_{i''}=-x_{j'}\neq 0$ for (1b) and $x_{i'}=x_{i''}=x_{j'}\neq 0$ for (2b). If $\tau$ does not have zeros as coordinates, then $s_\al(\mu)$ and $s_\be(\mu)$ belong to different $W_Q$-orbits, a contradiction. Therefore, $\tau$ must contain a zero coordinate.

Observe that $s_\al$ and $s_\be$ are not $W_Q$-conjugate
(if they are conjugate by an element of $W_Q$, then $j'$, $j''$ have to be in the same tuple as well as $i'$, $i''$). Moreover, if $\tau$ contains a zero coordinate and $x_{j''}=0$, $x_{i'}=x_{i''}=-x_{j'}\neq 0$ (resp. $x_{j''}=0$, $x_{i'}=x_{i''}=x_{j'}\neq 0$), then $(i'j')(\mu)$ (resp. $\widetilde{(i'j')}(\mu)$) and $\widetilde{(i'',j'')}(\mu)$ belong to the same $W_Q$-orbit but are not $W_Q$-conjugate.
Moreover, if $x_i>0$ (it is always possible to take $\mu$ so that this is true), then $\mu < (i'j')(\mu)$ 
(resp. $\mu < \widetilde{(i'j')}(\mu)$) and $\mu < \widetilde{(i'',j'')}(\mu)$.

Summarizing we obtain the following

\begin{prop}
In type $D$, a double moment graph $\QGP$ is not closed
if and only if both $\Theta'=\Theta_P,\Theta_Q$ satisfy
\[
{\rm (i)}~\;\al_{n-1},\al_n \in \Theta'\;\text{ and }\;{\rm (ii)}~\Theta'\neq \{\al_s,\al_{s+1},\ldots,\al_n\}\text{ for any }s\le n-1.
\]
\end{prop}

\begin{proof} If condition~(i) for $\Theta_Q$ holds, then the tuple $\tau$ above contains at least two coordinates.
Condition~(ii) for $\Theta_Q$ says that some other tuple contains at least two coordinates, so the case (2c) above may appear.
Condition~(i) for $\Theta_P$ is equivalent to the condition that $\mu$ has at least two zero coordinates. Permuting coordinates we may assume
that one zero coordinate belongs to the last tuple $\tau$ and another one belongs to a different tuple $\tau'$.
Condition~(ii) for $\Theta_P$ means that $\mu$ contains two non-zero coordinates such that one belongs to $\tau$ and another one to $\tau'$.
\end{proof}

The techniques used to treat non-simply laced cases is very similar, though requires more computations in type $B$.
We refer to the Appendix for all the details. Below we only state the final results.

\paragraph{\bf Type $B$.}
We realize a root system of type $B$ as a subset of vectors in $\R^n$, $n\ge 2$
\[
\Sigma=\{\pm(e_i+e_j),\pm(e_i-e_j),\pm e_i\mid i\neq j,\; i,j=1,\ldots,n\},
\]
where $e_1,\ldots,e_n$ are the standard basis vectors.
We fix a set of simple roots 
\[\Theta=\{\al_1=e_1-e_{2},\ldots,\al_{n-1}=e_{n-1}-e_n,\al_n=e_n\}\]
and its subsets $\Theta_P$ and $\Theta_Q$.

The Weyl group $W$
acts by permutations and sign changes, and the
reflections are given by usual transpositions, by signed transpositions, and by sign changes
\[(ij)\colon e_i\mapsto e_j,\; e_j\mapsto e_i\quad 
\widetilde{(ij)}\colon e_i\mapsto -e_j,\; e_j\mapsto -e_i,\quad \widetilde{(i)}\colon e_i\mapsto -e_i.\]
The dominant weights are the points whose coordinates are nonnegative integers and form a non-increasing sequence.

\begin{ex}
For the Weyl group $W=\langle s_1=(12),s_2=\widetilde{(2)}\rangle$ of type $B_2$ set $\Theta_P=\{\al_2\}$ and $\Theta_Q=\{\al_1\}$. Then 
\[W^P=\{e,s_1,s_2s_1,s_1s_2s_1\}.
\]
Set $\theta=e_1=\al_1+\al_2$ to be the dominant vector with stabilizer $W_P$. Then in terms of the $W$-orbit of $\theta$ we get
\[
W^P=\{(1,0),(0,1),(0,-1),(-1,0)\}.
\]
As for double cosets we obtain ${}^Q W^P=\{e, s_2s_1\}=\{(1,0),(0,-1)\}$. We have
\[
\widehat{s_{e_1+e_2}}=\widehat{s_2s_1s_2}=s_2s_1\text{ and }\widehat{s_{e_1}}=\widehat{s_1s_2s_1}=s_2s_1.
\]
Observe that the roots $e_1$ and $e_1+e_2$ belong to different $W_Q$-orbits.
\end{ex}

\begin{prop}\label{prop:typeB}
In type $B$, a double moment graph $\QGP$ is closed
if and only if 
one of the following is satisfied 
\begin{itemize}
\item There exist numbers $p$ and $q$ ($1 \le p, q \le n$)
such that 
$\Theta_P=\{\al_p,\ldots, \al_n\}$, 
$\Theta_Q=\{\al_q,\ldots,\al_n\}$.
\item 
$\al_n \notin \Theta_P$ and $\al_n \notin \Theta_Q$.
\item Either $\Theta_P$ or $\Theta_Q$ coincides with $\Theta$.
\item Either $\Theta_P$ or $\Theta_Q$ is empty.
\end{itemize}
\end{prop}

\begin{rem}
In other words, in type $B$, a double moment graph $\QGP$ is closed
if and only if 
one of the following is satisfied 
\begin{itemize}
\item Both $\Theta_P$ and $\Theta_Q$ are subsystems of type $B$ (for this we need 
to specify that the subsystem generated by $\al_n$ only is called a subsystem of type $B_1$, 
while the subsystem generated by any other simple root alone is not called a subsystem of type $B_1$).
\item 
Both $\Theta_P$ and $\Theta_Q$ do not contain a type $B$ subsystem,
keeping in mind the same remark about subsystems of type $B_1$.
\item Either $\Theta_P$ or $\Theta_Q$ coincides with $\Theta$.
\item Either $\Theta_P$ or $\Theta_Q$ is empty.
\end{itemize}
\end{rem}

\paragraph{\bf Type $C$.} 
The root system of type $C_n$ is dual to the root system of type $B_n$.
This duality induces a bijection between the roots systems and an isomorphism between the respective Weyl groups which identifies
simple roots/reflections and preserves their numbering.
Moreover, it identifies the reflection along a root $\al$ with the reflection
along a multiple of $\al$, i.e., with the same reflection;
the bijection between the root systems is equivariant 
under the Weyl group action.

Condition \ref{cond:transonedges} of Corollary~\ref{lem:clos} is stated only in terms of the Weyl group, 
its subgroups generated by some simple reflections, its action on the ambient space of the root system and 
on the root system itself, and of reflections along individual roots. 
Therefore, by Corollary \ref{lem:clos}, Proposition \ref{prop:typeB} can be stated in exactly the same way for type C:

\begin{prop}
In type $C$, a double moment graph $\QGP$ is closed
if and only if 
one of the following is satisfied 
\begin{itemize}
\item There exist numbers $p$ and $q$ ($1 \le p, q \le n$)
such that 
$\Theta_P=\{\al_p,\ldots, \al_n\}$, 
$\Theta_Q=\{\al_q,\ldots,\al_n\}$.
\item 
$\al_n \notin \Theta_P$ and $\al_n \notin \Theta_Q$.
\item Either $\Theta_P$ or $\Theta_Q$ coincides with $\Theta$.
\item Either $\Theta_P$ or $\Theta_Q$ is empty.\qed
\end{itemize}
\end{prop}

\paragraph{\bf Type $G_2$.} 
The root system of type $G_2$ can be constructed as the union of two root systems of type $A_2$, one of which
is obtained from the other one by stretching $\sqrt{3}$ times and by rotating by $\pi/6$.
The following remark follows directly from this construction and the fact that the angles 
between roots in $A_2$ are all possible multiples of $\pi/3$.
\begin{rem}\label{rem:g2orthogonal}
For each root $\al \in \Sigma$ there exists a root $\be \in \Sigma$ orthogonal to $\al$, 
and the length of $\be$ is different from the length of $\al$.\qed
\end{rem}
There are two simple roots, one short (denote it by $\al_1$) and one long (denote it by $\al_2$).
The angle between them is $5\pi / 6$. The Weyl group consists 
of all reflections along the roots and of rotations by multiples of $\pi/3$. 
It acts transitively on all long roots and (separately) on all short roots.

\begin{prop}
In type $G_2$, 
the double moment graph $\QGP$ is closed
if and only if
one of the following is satisfied
\begin{itemize}
\item Either $\Theta_P$ or $\Theta_Q$ coincides with $\Theta$,
\item Either $\Theta_P$ or $\Theta_Q$ is empty.
\end{itemize}
\end{prop}

\begin{proof}
If $\Theta_P$ or $\Theta_Q$ coincides with $\Theta$ or is empty, then it follows from 
Example~\ref{ex:pqnothing} and Example~\ref{lem:peverything}.

Suppose that neither $\Theta_P$ nor $\Theta_Q$ coincides with $\Theta$, and that neither $\Theta_P$ nor $\Theta_Q$ is empty.
This means that each of the systems $\Theta_P$ and $\Theta_Q$ contains exactly one simple root.
Denote these simple roots by $\gamma$ and $\de$, respectively, i.e. $\Theta_P=\{\gamma\}$
and $\Theta_Q=\{\de\}$.

By Remark \ref{rem:g2orthogonal}, there exists a root orthogonal to $\gamma$, denote it temporarily by $\gamma'$.
Then $-\gamma'$ is also a root orthogonal to $\gamma$. Also denote the simple root different from $\gamma$ by $\gamma''$.
Then the angles between $\gamma''$ and $\pm \gamma'$ are $\pi/3$ and $2\pi/3$ (in some order), and 
without loss of generality we may assume that the angle between $\gamma'$ and $\gamma''$ is $\pi/3$.
Then $\gamma'$ is a dominant weight.
The only simple root orthogonal to $\gamma'$ is $\gamma$, so we can set $\theta=\gamma'$, and 
the stabilizer of $\theta$ will be exactly $W_P$. The orbit $W\theta$ consists of all roots of the 
same length as $\theta$.

By Remark \ref{rem:g2orthogonal} again, there exists a root orthogonal to $\de$, denote it by $\al$.
By a similar argument, without loss of generality, $\al$ is also a dominant weight, 
and then it is also a positive root.

The orbit $W\theta$ contains at most two roots proportional to $\de$ and at most two roots orthogonal to $\de$.
(Moreover, in fact, it cannot contain both roots proportional to $\de$ and roots orthogonal to $\de$, 
because these roots have different lengths.) 
In total, there are 6 roots in $W\theta$, so we can take $\mu \in W\theta$ that is not orthogonal to $\de$ and 
is not a multiple of $\de$. Again, without loss of generality, $\mu$ is a positive root. Denote also $\be=\mu$.

We are going to use Corollary~\ref{lem:clos}.
First, note that $s_{\al}$ and $s_{\de}$ are two commuting reflections since $\al$ and $\de$ are orthogonal.
Moreover, their composition is the rotation by $\pi$, i.e. the change of sign. In particular, 
$s_{\de}s_{\al}\mu=-\mu$. But we have set $\be=\mu$, so $s_{\be}\mu=-\mu=s_{\de}s_{\al}\mu$,
and $s_{\be}\mu \in W_Q s_{\al}\mu$.

Second, the orbit $W_Q \mu$ consists of two roots: $\mu$ and $s_{\de}\mu$. These two roots are different since 
we chose $\mu$ so that it is not orthogonal to $\de$. 
In other words, $s_{\de}\mu-\mu$ is a nonzero multiple of $\de$.
Similarly, $\mu$ is not a multiple of $\de$, so $\mu$ is not orthogonal to $\al$. 
Therefore, $\mu$ and $s_{\al}\mu$ are two different roots, and $s_{\al}\mu-\mu$
is a nonzero multiple of $\al$.

Therefore, we cannot have $s_{\al}\mu=s_{\de}\mu$. 
Indeed, otherwise $s_{\al}\mu-\mu$ would be equal to $s_{\de}\mu-\mu$, 
and $\al$ and $\de$ are orthogonal, a contradiction.

So, $s_{\al}\mu \ne \mu$, $s_{\al}\mu \ne s_{\de}\mu$, and $s_{\al}\mu \notin W_Q \mu$.
Therefore, $W_Q s_{\al} \mu$ and $W_Q\mu$ are two different $W_Q$-orbits, and $s_{\al}\mu, s_{\be}\mu \notin W_Q \mu$.
Finally, $\al$ and $\de$ are orthogonal, so for every $w\in W_Q$ we have $w(\al)=\al$, 
and it is impossible to find $w \in W_Q$ such that $w(\al)=\be$.
Clearly, $(\mu, \be) > 0$ since $\mu = \be$. Also, $(\mu, \al) > 0$ since $\mu$ is not 
orthogonal to $\al$, $\mu$ is a positive root and $\al$ is a dominant weight.
By Corollary~\ref{lem:clos}, 
the closure of a double moment graph $\QGP$
cannot coincide with $\GP$.
\end{proof}


\section{Global sections of structure sheaves}\label{sec:globsec}

In the present section we describe $W$-invariants of the $T$-equivariant cohomology of the product $\hh_T(G/Q\times G/P)$ as modules of global sections of structure sheaves on moment graphs. 

\paragraph{\it Oriented cohomology sheaves.}
Let $G$ be a split semisimple linear algebraic group over $k$, let $T$ be its split maximal torus. Let $\hh_G(-)$ (resp. $\hh_T(-)$) be a $G$-equivariant (resp. $T$-equivariant) algebraic oriented cohomology theory. We assume that $\hh_T(pt)$ is naturally a $\hh_G(pt)$-module via the forgetful map.
Let $\Sigma=\Sigma^+\amalg \Sigma^- \supset \Theta$ be the associated root datum and let $\Gg=((V,\le),l\colon E\to \Sigma^+)$ be a moment graph. Observe that any character $\la \colon T \to \mathbb{G}_m$ corresponds to a $T$-equivariant line bundle $\Ll_\la$ over $pt$ and, hence, it gives the characteristic class $c_1^{\hh}(\Ll_\la) \in \hh_T(pt)$.

We extend the notion of a moment graph sheaf of \cite{BMP} as follows

\begin{dfn}
 An $\hh$-sheaf $\Ff$ on the moment graph $\Gg$ is given by the data
\[
(\{\Ff^{x}\},\{\Ff^{x\to y}\}, \{\rho_{x,x\to y},\rho_{y,x\to y}\}),
\] 
where for all $x\in V$ and for all $x\to y \in E$
\begin{itemize}
\item[(SH1)] $\Ff^{x}$ is an $\hh_G(pt)$-module,
\item[(SH2)]  $\Ff^{x\to y}$ is an $\hh_T(pt)$-module such that $c_1^{\hh}(\Ll_{l(x\to y)}) \cdot \Ff^{x\to y}=(0)$,
\item[(SH3)] the maps $\rho_{x,x\to y}\colon \Ff^x\to \Ff^{x\to y}$ and $\rho_{y,x\to y}\colon \Ff^y\to \Ff^{x\to y}$ are homomorphisms of $\hh_G(pt)$-modules. 
\end{itemize}
\end{dfn}

\begin{ex}\label{ex:CHmom}
Let $\hh(-)=\CH(-;\Q)$ be the Chow theory with rational coefficients. Then $\hh_T(pt)$ is the polynomial ring in simple roots over $\Q$ denoted simply by $S$ (cf. Example~\ref{ex:chr}). The Weyl group $W$ of $G$ acts naturally on $S$. Its subring of invariants $S^W$ can be identified with $\hh_G(pt)$. The following data define a $\CH(-;\Q)$-sheaf on the parabolic moment graph $\GP$ of Example~\ref{ex:parabolic}:
\begin{itemize}
\item[(SH1)] $\Ff^{x}=S$, for all $x\in W^P$, 
\item[(SH2)] $\Ff^{x\to y}=S/\al S$ for all $x \to y=\overline{s_\al x}\in E$ and $\al\in \Sigma^+$,
\item[(SH3)] $\rho_{x,x\to y}=\rho_{y,x\to y}$ is simply the quotient map $S\to S/\al S$. 
\end{itemize}
This sheaf is called \emph{sheaf of rings} in \cite{BMP}, and  \emph{structure sheaf} in \cite{F1} and subsequent papers.
\end{ex}

\begin{ex}
Let $\hh(-)=K^0(-)[\be^{\pm 1}]$, $\be=1$ be the $K$-theory. Then $\hh_T(pt)$ is the group ring $Z=\Z[T^*]$  of the characters of $T$ (cf. Example~\ref{ex:kth}). 
Its subring of invariants $Z^W$
can be identified with $\hh_G(pt)$. Each character $\la$ defines the first characteristic class $c_1^{\hh}(\Ll_\la)=1-e^{-\la}$. The following data define a $K^0$-sheaf on $\GP$:
\begin{itemize}
\item[(SH1)] $\Ff^{x}=Z$, for all $x\in W^P$, 
\item[(SH2)] $\Ff^{x\to y}=Z/(1-e^{-\al}) Z$ for all $x\to y=\overline{s_\al x}\in E$ and $\al\in \Sigma^+$,
\item[(SH3)] $\rho_{x,x\to y}=\rho_{y,x\to y}$ is simply the quotient map $Z\to Z/(1-e^{-\al}) Z$. 
\end{itemize}
\end{ex}

\begin{ex}\label{ex:parsheaf} 
More generally, for an arbitrary oriented cohomology theory $\hh$ the following data define an $\hh$-sheaf $\Ff$ on  $\GP$:

for all $u\in W^P$ and for all $u\to v \in E$ we set
\begin{itemize}
\item[(SH1)]  $\Ff^{u}=S$, where $S=\hh_T(pt)^\wedge$,
\item[(SH2)] $\Ff^{u\to v}=S/x_{l(u\to v)}S$,
\item[(SH3)] $\rho_{u,u\to v}=\rho_{v,u\to v}$
to be the quotient map $S \to S/x_{l(u\to v)}S$;
\end{itemize}
Similar to Example~\ref{ex:CHmom} we call such $\Ff$ the (parabolic) structure $\hh$-sheaf on $\GP$ and denote it by $\OP$.
\end{ex}

\begin{ex}\label{prop:motsheaf} In the notation of Section~\ref{sec:doublemom}
consider a double moment graph \[\QGP=(({}^QW^P,\le),l\colon {}^QE^P\to \Sigma^+)\] for $G$, $P$ and $Q$. 
The following data define an $\hh$-sheaf  $\Ff$ on $\QGP$:

for all $u\in {}^QW^P$ and for all $u\to u' \in {}^QE^P$ we set
\begin{enumerate}
\item[(SH1)] $\Ff^{u}=S^{W_u}$, where $S=\hh_T(pt)^\wedge$ and $W_u$ corresponds to $\Theta_u$,
\item[(SH2)] $\Ff^{u\to u'}=S/c_1^{\hh}(\Ll_{l(u\to u')})S$,
\item[(SH3)] As for the maps $\rho_{u,u\to u'}$ and $\rho_{u',u\to u'}$ we set
\[\rho_{u,u\to u'}=\textrm{can}\circ \iota\; \text{ and }\;\rho_{u',u\to u'}:=\textrm{can}\circ \iota_{w},\; 
\text{ where}\] 
$\textrm{can}\colon S\to S/c_1^{\hh}(\Ll_{l(u\to u')})S$ is the canonical quotient map, \\
$\iota\colon S^{W_{u}}\hookrightarrow S$ is the canonical inclusion, and \\
$\iota_w\colon S^{W_{u'}}\hookrightarrow S$ is a twisted inclusion $s\mapsto w(s)$ corresponding to the element $w\in W_Q^{u'}$ uniquely determined by the double parabolic decomposition $s_{l(u\to u')}u =wu'v$, $v\in W_P$.

\end{enumerate}
We call such $\Ff$ a (double) structure $\hh$-sheaf and denote it $\QOP$.
\end{ex}

The following is a natural generalization to our setting of the module of local sections of a moment graph sheaf from \cite[\S 1.3]{BMP}.

\begin{dfn}\label{dfn:globsect}
Let $I$ be a subset of $V$, the module of local sections of an $\hh$-sheaf $\Ff$ on a moment graph $\Gg$ is defined as
\[
\Gamma(I,\Ff):=\left\{ (m_x)\in \prod_{x\in I} \Ff^{x} \mid
\begin{array}{c} 
\rho_{x,x\to y}(m_x)=\rho_{y,x\to y}(m_y)\\ \forall x\to y\in E ,\, x,y\in I                                                                                                                 
                                                                                                   \end{array}   \right\}.
\]
We write $\Gamma(\Ff)$ for the module of global sections $\Gamma(V,\Ff)$ of the $\hh$-sheaf $\Ff$.
\end{dfn}

\begin{ex} From the main result of \cite{KK86} it follows that 
\[
\Gamma(\OP)\simeq \CH_T(G/P;\Q),
\]
where $\OP$ is the parabolic structure  $\CH(-;\Q)$-sheaf.
From \cite{KK90} it follows that
\[
\Gamma(\OP)\simeq K^0_T(G/P),
\]
where $\OP$ is the parabolic structure  $K_0$-sheaf.
For an arbitrary $\hh$ we have \cite[Thm.~11.9]{CZZ}
\[\Gamma(\OP)\simeq \hh_T(G/P).\]
\end{ex}

Our main result is the following

\begin{thm}\label{thm:main}  Let $G$ be a split semisimple linear algebraic group over a field of characteristic $0$. Let $P$, $Q$ be standard parabolic subgroups containing a split maximal torus $T$. Let $W$ be the Weyl group.
Let $\hh_G(-)$ be an oriented equivariant cohomology theory in the sense of \cite{HML}. Let $S$
be the $T$-equivariant coefficient ring $\hh_T(pt)$ which is $\Sigma$-regular in the sense of \cite{CZZ}.
Let $\QGP$ be the double moment graph with the (double) structure $\hh$-sheaf $\QOP$.

Then the ring of $W$-invariants $\hh_{T}(G/Q\times G/P)^W$ is isomorphic to
\begin{itemize}
\item[(A)]
$
{}^{W_Q}R_{W_P}=\left\{ (b_v)\in \bigoplus_{v\in W^P} S^{W_v} \middle| 
\begin{aligned}
{\rm (i)} &\;   b_v-b_{\overline{s_\be v}}\in x_\be S\text{ for all }s_\be\in W, \text{ and} \\
{\rm (ii)} &\;  b_{\overline{s_\al v}}=s_\al(b_v)\text{ for all }s_\al\in W_Q
\end{aligned}
\right\}
$
and it is also isomorphic to
\item[(B)]
$
{}^Q\!A^P:=
\left\{(c_u)\in \bigoplus_{u\in {^Q W^P}} S^{W_{u}} \middle|
\begin{array}{c}
c_{u}- w(c_{u'})\in x_{\al} S\\
\hbox{for all }\al\in\Sigma^+, w\in W_Q \\
\hbox{ s. t. } u\in s_{\al}wu' W_P
\end{array}\right\}$,
\end{itemize}
where the product structure is given by the coordinate-wise multiplication.

Moreover, ${}^Q\!A^P\subseteq \Gamma(\QOP)$, where the equality holds if $\QGP$ is closed.
 \end{thm}
 
\begin{proof}
(A) By Proposition~\ref{prop:winvis} we have $\hh_{T}(G/Q\times G/P)^W\simeq \hh_{T}(G/P)^{W_Q}$.
Using \cite[Theorem~11.9]{CZZ}, we can identify $\hh_T(G/P)$ with the following $S$-subalgebra of $S_{W/W_P}^\star$
\begin{equation*}\tag{*}
R_{W_P}=
\left\{
(b_{v}) \in \bigoplus_{v\in W^P}S \mid
b_{v}- b_{\overline{s_\be v}}\in x_{\be} S,\; \forall s_\be \notin vW_Pv^{-1}
\right\}
\end{equation*}
Observe that
the divisibility does not depend on a choice of representative of the coset $s_\al v$ as there are no multiple edges in the respective (parabolic) moment graph.

The group $W$ acts on $Q_W^*$ via the $\odot$-action given by
\[
w\odot b_v f_v=w(b_v) f_{wv}\quad \text{or, equivalently,}\quad w\odot (b_{v})=(w(b_{w^{-1}v})).
\]
This induces an action of $W$ on $S_{W/W_P}^\star$ via $w\odot b_{v}f_{v}=w(b_{v})f_{\overline{wv}}$
which restricts to an action on $R_{W_P} \subset S_{W/W_P}^\star$.

Now, $(b_v)\in R_{W_P}$ is $W_Q$-invariant if and only if for all $\al\in \Theta_Q$ and for all $v\in W^P$ we have $s_\al \odot (b_v)=(s_\al (b_{\overline{s_\al v}}))=(b_v)$. In other words, $\overline{s_\al v}=v$ and
\begin{equation*}\tag{**}
b_{\overline{s_\al v}}=s_\al(b_v),\quad \forall s_\al\in W_Q,
\end{equation*} 
Observe that $\overline{s_\al v}=v$ if and only if $s_\al \in v W_P v^{-1}$ $\Longleftrightarrow$ $\al \in \Theta_v=\Theta_Q \cap v(\Sigma^+_P)$. The latter implies that
\[
b_v=b_{\overline{s_\al v}}=s_\al(b_v),\; \forall \al \in \Theta_v,\text{ hence, }b_v\in S^{W_v}.
\]
Combining (*) and (**) we prove the isomorphism (A).

(B) Observe that $\al$ is uniquely determined by the class of $w$ modulo $W_{u'}$, i.e., if $u\in s_{\al}wu' W_P$ and $u\in s_\be w' u'W_P$,
where $w$ and $w'$ belong to the same coset in $W_Q/W_{u'}$, then $\al=\be$.

We define a map
\[
\psi\colon {}^Q\!A^P\rightarrow \bigoplus_{v\in W^P} S^{W_{v}},\quad \psi(c_u):=(w(c_u))_{wu},\; w\in W_Q^u.
\]
This map is clearly an injective homomorphism of $S^W$-algebras, so if we show that its image is ${}^{W_Q}R_{W_P}$ we are done.

Let $v=wu \in W^P$ be a parabolic decomposition.
Since $s_\al v=s_\al w u$, $s_\al\in W_Q$ is in the same double coset as $v$, 
$\overline{s_\al v}=w' u$ is the parabolic decomposition of $\overline{s_\al v}$, where $w'$ is the minimal representative of the coset of $s_\al w$ in
$W_Q/W_u$. We then have $b_{\overline{s_\al v}}=w'(c_u)=s_\al(w(c_u))=s_\al(b_v)$.
In other words, the image of $\psi$ is $W_Q$-invariant, i.e. (ii) holds.

Consider parabolic decompositions $v=wu$ and $\overline{s_\be v}=w'u'$, where $u,u'\in {}^Q W^P$,
$w$ is a minimal representative in $W_Q/W_u$ and $w'$ is a minimal representative in $W_Q/W_{u'}$. Then $b_v=w(c_u)$,
$b_{\overline{s_\be v}}=w'(c_{u'})$ and the condition (i) becomes equivalent to
$w(c_u)-w'(c_{u'}) \in x_\be S$ and, hence, to
$c_u - w^{-1}w'(c_{u'}) \in x_{w^{-1}(\be)}S$.
Since $w'u'z=s_\be v=s_\be w u$ for some $z\in W_P$, we obtain $u=w^{-1}s_\be w'u'z=s_{w^{-1}(\be)}w^{-1}w' u'z$.
Substituting $w^{-1}w'$ by $w$ we obtain the condition on $(c_u)$ in the definition of ${}^Q\! A^P$, Hence, (i) holds.

We have shown that $\textrm{Im}(\psi)\subseteq {}^{W_Q}R_{W_P}$. 

Vice versa, given an element $(b_v)_{v\in W^P}$, 
one can obtain an element in ${}^Q\! A^P$ by projecting onto the ${}^QW^P$-components: 
Set $c_u=w^{-1}(b_v)$, where $v=wu$ is the parabolic decomposition.
If 
$u, u'\in {}^QW^P$, $w\in W_Q$, and $\al\in\Sigma^+$ are such that $u\in s_{\al}wu'W_P$, then $c_u-w(c_{u'})=c_u-c_{wu'}\in x_\al S$ (since $c_{u'}\in S^{W_{u'}}$ we may assume $wu'\in {}^Q W^P$) and so $(c_u)_{u\in {}^QW^P}\in {}^Q\! A^P$. 

Finally, by Example~\ref{prop:motsheaf} and Definition~\ref{dfn:globsect} we have
\[\Gamma(\QGP)=
\left\{(c_u)\in \bigoplus_{u\in {^Q W^P}} S^{W_{u}} \middle|
\begin{array}{c}
c_{u}- w(c_{u'})\in x_{\al} S\\
\hbox{for all }u\to u'\in {}^QE^P,\hbox{ where } \al=l(u\to u')\\
\hbox{and } w\in W_Q^u \hbox{ is determined by }s_\al u\in wu'W_P.
\end{array}\right\}.\]
By (B) ${}^Q\!A^P$ is a submodule of $\Gamma(\QGP)$. 
If $\QGP$ is closed, then by Lemma~\ref{lem:geomorb} any edge $v\stackrel{\be}\to \overline{s_\be v}$ of the parabolic moment graph $\GP$
is obtained from some edge $u\to \overline{s_\al u}$, $u\in {}^QW^P$, $\al=l(u\to \widehat{s_\al u})$ by applying an element $w$ of $W_Q$. The divisibility condition $b_v-b_{\overline{s_\be v}}\in x_\be S$
of (A) then can be rewritten as 
\[b_{wu}-b_{\overline{s_{w(\al)} wu}}=w(b_u)-w(b_{\overline{s_\al u}}) \in x_{w(\al)}S.\]
So it is redundant with respect to 
$b_u-b_{\overline{s_\al u}} \in x_{\al}S$ or to $c_u-w(c_{u'})\in x_\al S$, with $u'=\widehat{s_\al u}$, $u\in s_{\al}wu' W_P$. Hence, ${}^Q\!A^P=\Gamma(\QGP)$.
\end{proof}

\begin{ex} Let $G=SL_3$  and let $P_1$ and $P_2$ denote the two parabolic subgroups
\[
P_1=\left(
\begin{array}{ccc}
\ast&\ast&\ast\\
\ast&\ast&\ast\\
0&0&\ast\\
\end{array}
\right),
\quad
P_2=
\left(
\begin{array}{ccc}
\ast&\ast&\ast\\
0&\ast&\ast\\
0&\ast&\ast
\end{array}
\right)
\]
Let $s_1=(12)$, $s_2=(23)$ be the two simple reflections of $W=S_3$ corresponding to the two parabolic subgroups above.
Let $\hh=\CH$ be the Chow theory. We then have $S:=\Z[e_1,e_2]$ (the polynomial ring in two variables with integer coefficients). 

The minimal length representatives for $S_3/\langle s_1 \rangle$ are $\{e,s_2,s_1s_2\}$ and the moment graph description of $\CH_T(SL_3/P_1)$ is
\[
\left\{(z_e,z_{s_2},z_{s_1s_2})\in S\oplus S\oplus S \mid \,
\begin{array}{c}z_e-z_{s_2}\in\al_2 S,\,
z_{s_2}-z_{s_1s_2}\in\al_1 S\\
z_e-z_{s_1s_2}\in(\al_1+\al_2)S
\end{array}
\right\}.\]
Here $s_1s_2=s_{s_2(\al_1+\al_2)}s_2$, where $s_2(\al_1+\al_2)=\al_1$.

The action of $S_3$ on $\CH_T(SL_3/P_1)$ is determined by
\begin{equation}\label{SL3s1Action}
s_1((z_e,z_{s_2},z_{s_1s_2}))=(s_1(z_e),s_1(z_{s_1s_2}),s_1(z_{s_2})),
\end{equation}
\begin{equation}\label{SL3s2Action}
s_2((z_e,z_{s_2},z_{s_1s_2}))=(s_2(z_{s_2}),s_2(z_e),s_2(z_{s_1s_2})).
\end{equation}
\end{ex}

By equation (\ref{SL3s1Action}), an element $(z_e,z_{s_2},z_{s_1s_2})\in \CH_T(SL_3/P_1) $ is $s_1$-invariant if and only if
\[
z_e=s_1(z_e), \, z_{s_2}=s_1(z_{s_1s_2}).
\]
(The third condition $z_{s_1s_2}=s_1(z_{s_2})$ is clearly redundant).  We deduce that
\[
{}^{s_1}\CH_T(SL_3/P_1)=\]
\[
\left\{(z_e,z_{s_2},s_1(z_{s_2}))\in S^{s_1}\oplus S\oplus S\mid z_e-z_{s_2}\in\al_2 S,  \, z_e-s_1(z_{s_2})\in(\al_1+\al_2)S)
\right\}.
\]
We notice that the second congruence is equivalent to the first one, since $z_e\in S^{s_1}$ and $s_1(\al_2)=\al_1+\al_2$. 
Thus we have an isomorphism of $S^{s_1}$-modules:
\[{}^{s_1}\CH_T(SL_3/P_1)\cong
\left\{
(z_e,z_{s_2})\in S^{s_1}\oplus S\mid z_e-z_{s_2}\in\al_2S 
\right\}.
\]
This description is encoded in the double moment graph:
\[
e\stackrel{\al_2}\longrightarrow s_2.
\]

We notice that $\{(1,1)$, $(\al_2(\al_1+\al_2),0)$, $(0, \al_2)\}$ is a basis for ${}^{s_1}\CH_T(SL_3/P_1)$ as a free $S^{s_1}$-module.

\begin{ex} Equation (\ref{SL3s2Action}) tells us that $(z_e,z_{s_2},z_{s_1s_2})\in \CH_T(SL_3/P_1)$ is $s_2$-invariant if and only if
\[
z_{s_2}=s_2(z_e)\quad\hbox{ and}\quad z_{s_1s_2}=s_2(z_{s_1s_2}), 
\]
from which, reasoning as in the previous case, we get
\[
{}^{s_2}\CH_T(SL_3/P_1)\cong
\left\{
(z_e,z_{s_1s_2})\in S\oplus S^{s_2}\mid z_e-z_{s_2s_1}\in (\al_1+\al_2)S
\right\}.
\]
Such a space is an $S^{s_2}$-module with a basis $\{(1,1),(\al_1+\al_2,0), (0,\al_1(\al_1+\al_2))\}$.
The corresponding double moment graph is then
\[
e\stackrel{\al_1+\al_2}\longrightarrow s_1s_2.
\]
\end{ex}

\section{Correspondence product and equivariant motives}\label{sec:motives}

In the present section we interpret the correspondence
product in the category of equivariant $\hh$-motives in moment graph terms (see Proposition~\ref{prop:corr}).
In particular, we show that endomorphisms of the equivariant motive of the total flag variety of type $A$
can be identified with global sections of the structure {\tt h}-sheaf on the total
double moment graph (see Corollary~\ref{cor:typeA}).

\paragraph{\it Correspondence product.}
The $\hh_G(pt)$-module $\hh_G(G/B\times G/B)$ can be endowed with another product structure the so called correspondence product. Namely, given $f,g\in \hh_G(G/B\times G/B)$ we set
\[
f\circ g:=pr_{1,3*}(pr_{1,2}^*(f)\cdot pr_{2,3}^*(g)),\] 
where $pr_{i,j}\colon (G/B)^3 \to (G/B)^2$ are projections to the respective components and $pr^*$, $pr_*$ denote the induced pull-backs and push-forwards in the theory $\hh$.
Since $G/B\times G/B$ is a $G$-equivariant cellular space over $G/B$ (via the orbit stratification), and $\hh_G(G/B)=\hh_T(pt)=\hh_T(G/B)^W$, we have 
\[\hh_G(G/B\times G/B)=\hh_{G|T}(G/B\times G/B)=\hh_T(G/B\times G/B)^W
\] 
as $\hh_G(pt)$-modules. Following the notation of Theorem~\ref{thm:main}.(B) we denote the latter by ${}^B\! A^B$.
By \cite[Theorem~6.2]{NPSZ}
the ring ${}^B\! A^B$ with respect to the correspondence product `$\circ$'
is a (non-commutative) ring isomorphic to $\DF$. Moreover, by the K\"unneth isomorphism \eqref{eq:kunn} it is also isomorphic to
the endomorphism ring
$\End_{\DF}(\DFd)$ of $\DF$-modules (w.r. to the $\odot$-action).

More generally, given parabolic subgroups $P$, $Q$, $H$ of $G$ 
there is the $\hh_G(pt)$-bilinear map
\[
\circ \colon  \hh_G(G/Q\times G/P) \otimes \hh_G(G/P\times G/H) \to \hh_G(G/Q\times G/H),
\]
defined by the correspondence product 
\[(\phi,\psi)\mapsto \psi\circ \phi=pr_{Q,H*}(pr_{P,H}^*(\psi)\cdot pr_{Q,P}^*(\phi)),\]
where $pr_{Q,P}\colon G/Q\times G/P\times G/H \to G/Q\times G/P$ is the corresponding projection.
As the projections are $W$-equivariant,
it restricts to the $W$-invariants and, hence, it induces
the $S^W$-bilinear pairing
\[
{}^QA^P \otimes {}^PA^H \to {}^QA^H,\; \text{ where }
{}^QA^P=\hh_T(G/Q\times G/P)^W.
\]

\paragraph{\it Equivariant motives}
Let $\mathfrak{M}_G$ denote the full additive subcategory of the category of $G$-equivariant $\hh$-motives generated by the motives of varieties $G/P$ for all parabolic $P$ (we refer to \cite[\S2]{GV} for definitions and basic properties of this category). Recall only that morphisms in this category are cohomology groups $\hh_G(X\times Y)$, where $G$ acts diagonally on the product of smooth projective $G$-varieties $X$ and $Y$ and the composition is given by the correspondence product.
Following \cite{CNZ, NPSZ} let $\mathfrak{M}_{G|T}$ denote
the full additive subcategory of the respective category of relative equivariant $\hh$-motives, where the morphisms are given by images $\hh_{G|T}(X\times Y)$ of forgetful maps. Similarly, let $\mathfrak{M}_T^W$ denote the category
of $W$-invariant equivariant motives with morphisms $\hh_T(X\times Y)^W$ (observe that the correspondence product on $\mathfrak{M}_T$ is $W$-equivariant).

By the K\"unneth isomorphism \eqref{eq:kunn} the category $\mathfrak{M}_T^W$ is equivalent to the full additive subcategory of $\DF$-modules (equivalently, of $S_W$-modules) generated by $\DPd$ for all parabolics $P$.
There are induced functors
\[
\mathfrak{M}_G \stackrel{\Ff_1}\to \mathfrak{M}_{G|T} \stackrel{\Ff_2}\to \mathfrak{M}_T^W
\]
defined on 
morphisms by \[\hh_G(G/Q\times G/P)\to \hh_{G|T}(G/Q\times G/P) \hookrightarrow \hh_T(G/Q\times G/P)^W={}^Q\! A^P\]
so the functor $\Ff_2$
is faithful by definition.

\begin{rem}
All these functors become equivalences if $R$ contains $\Q$ or if the torsion index of $G$ is $1$ or if $G$ is simply-connected and $\hh=K_0$.
\end{rem}

\begin{rem}
For some theories $\hh$ (e.g. for $\CH$ and $K_0$) induced maps on endomorphisms rings have nilpotent kernels (this is essentially the Rost nilpotence theorem for $\hh$ discussed in \cite{CM, GV}).
So there is a lifting of idempotents with respect to $\Ff_1$.
In other words, if $\End_{\DF}(\DPd)={}^P\! A^P$ has no idempotents, then the $G$-equivariant $\hh$-motive of $G/P$ is indecomposable.
\end{rem}

\begin{rem}
For the Chow theory $\CH(-;\Z/p\Z)$, $p$ is a prime, and some groups $G$ (e.g. see \cite{Vi} for $PGL_p$), the functor $\Ff_2$ is an equivalence.
The latter is closely related to surjectivity of the map $\CH(BG) \to \CH(BT)^W$. 
Indeed, if these maps are surjective for Levi-parts of all parabolic subgroups of $G$, including $G$ itself, then $\Ff_2$ is an equivalence.
\end{rem}

\begin{dfn}
Consider the total flag variety $X=\amalg_P G/P$ (the disjoint union is taken over all parabolic subgroups). 
Consider the direct sum $\Os=\oplus_{P,Q} \QOP$ of double structure $\hh$-sheaves supported 
on the respective double moment graphs (see Example~\ref{prop:motsheaf}). We call $\Os$ the total double structure $\hh$-sheaf of the group $G$. 
\end{dfn}

Suppose $G$ is of type $A$. Then any double moment graph $\QGP$ for $G$ is closed by Proposition~\ref{prop:typeAc}.
As an immediate consequence of Theorem~\ref{thm:main}.(B) the module of global sections can be identified with
\[
\Gamma(\Os)=\oplus_{P,Q} \Gamma(\QOP)=\bigoplus_{P,Q} \hh_T(G/Q\times G/P)^W.
\]
Moreover, the correspondence product turns $\Gamma(\Os)$ into a (non-commutative) ring which is
isomorphic to the endomorphism ring of the $W$-invariant equivariant motive of the total flag~$X$, i.e., we obtain
\begin{cor}\label{cor:typeA} If $G$ is of type $A$, then $(\Gamma(\Os),\circ) \simeq \End_{\mathfrak{M}_T^W}([X])$.
\end{cor}

\begin{ex} Suppose $\hh=\CH(-;\Q)$. Then $\hh_T(G/Q\times G/P)^W\simeq S^{W_Q}\otimes_{S^W} S^{W_P}$ is a free
$S^W$-module of rank $r(Q,P)=|{}^QW^P|$. Therefore, 
\[(\hh_T(G/Q\times G/P)^W,\circ)\simeq M_{r(Q,P)}(S^W) \text{ and}\]
$\Gamma(\Os)$ is isomorphic to the ring of $r\times r$ matrices over $S^W$, where $r=\sum_{(Q,P)} r(Q,P)$.
\end{ex}

\paragraph{\it Product of global sections.}
Given $\phi\in {}^Q\!A^P$ and $\psi\in {}^P\!A^H$ so that $\phi$ and $\psi$ are global sections of the respective double structure sheaves $\QOP$ and ${}^P\Os^H$, we associate
a new element $\psi\circ \phi\in {}^Q\! A ^H$ given by the correspondence product 
which is a global section of the respective double structure sheaf ${}^Q \Os^H$.
We describe this product explicitly (in terms of the standard localization basis) as follows:

First, we extend the arguments of \cite[\S8]{CNZ} (were the case $P=Q$ was considered)
to describe the $Q$-module $\Hom_{Q_W} (Q_{W/W_Q}^*,Q_{W/W_P}^*)$.
Let $\{f_v \}$ denote the standard basis of the free $Q$-module $Q_{W/W_Q}^*$ or $Q_{W/W_P}^*$ (we use the same letter $f$ for these bases 
indexed by the minimal coset representatives). 
Since the $Q_W$-module (via the $\odot$-action) $Q_{W/W_Q}^*$ is generated by $f_e$ (we have $w\odot f_e=f_{\overline{w}}$ for any $w\in W$),
any homomorphism $\phi\in \Hom_{Q_W}(Q_{W/W_Q}^*,Q_{W/W_P}^*)$ is uniquely determined by its value on $f_e$ that is 
\[
\phi(f_e)=(b_v)=\sum_{v\in W^P}b_vf_v, \quad b_v\in Q.\]

Similarly, $\psi\in \Hom_{Q_W}(Q_{W/W_P}^*,Q_{W/W_H}^*)$ is also determined by its value at $f_e$.
Let 
\[
\psi(f_e)=(c_u)=\sum_{u\in W^H}c_uf_u, \quad c_u\in Q,
\]
 hence,
we have ($\phi$ and $\psi$ are $Q_W$-equivariant)
\[
\psi\circ \phi (f_e)=\psi(\sum_{v\in W^P}b_vf_v)=\sum_{v\in W^P}b_v\psi(f_v)=\sum_{v\in W^P}b_v\psi(v\odot f_e)=
\]
\[
\sum_{v\in W^P}b_v(v\odot \psi(f_e))=\sum_{v\in W^P} b_v \sum_{u\in W^H} v(c_u)f_{\overline{vu}}.
\]
Summarizing, we obtain the following

\begin{prop} \label{prop:corr}
The correspondence product (in terms of the standard localization basis) $(a_w)=(c_u)\circ (b_v)$ is given by
\[
\sum_{\{v\in W^P,\; u\in W^H\;\mid\; \overline{vu}=w\}} b_v v(c_u)=a_w,\;\text{ for all }\;w\in W^H.
\]
\end{prop}

Observe that tensoring with $Q$ induces
an embedding
\[
\hh_T(G/Q\times G/P)^W=\Hom_{\DF} (\mathbb{D}^{\star}_{F,Q},\DPd)\hookrightarrow \Hom_{Q_W} (Q_{W/W_Q}^*,Q_{W/W_P}^*).
\]
We now investigate its image.
Recall that (cf. the proof of Lemma~\ref{lem:Winvc}) $\mathbb{D}^{\star}_{F,Q}$ is a $\DF$-module generated by the class of a point 
$[pt_Q]=x_Qf_{e} \in S_{W/W_Q}^\star$, 
Therefore, any $\phi=(b_v)\in \Hom_{\DF} (\mathbb{D}^{\star}_{F,Q},\DPd)$ is determined by its value on $[pt_Q]$. On the other hand, $\phi([pt_Q])$ belongs to  $\DPd$ as an element of 
$S_{W/W_P}^\star \subset Q_{W/W_P}^*$ if it satisfies the criteria (i) of Theorem~\ref{thm:main}(A). Moreover, since
 $\phi$ is $Q_W$-linear, it has to satisfy 
\[
w\odot \phi(f_e) = \phi(w\odot f_e)=\phi(f_{\overline{w}})=\phi(f_e)\]
 for all $w\in W_Q$, which translates as
\[
w\odot \sum_{v\in W^P}b_vf_v =
\sum_{v\in W^P} w(b_v)f_{\overline{wv}}=\sum_{v\in W^P}b_{v} f_v \quad \text{for all }w\in W_Q.
\]
Observe that the latter gives the condition (ii) of Theorem~\ref{thm:main}(A). 
In other words, (ii)~guarantees that $(b_v)$ defines a $Q_W$-equivariant homomorphism $Q_{W/W_Q}^*\to Q_{W/W_P}^*$.

Combining these together we obtain

\begin{lem} The element (global section of the localized double structure sheaf)
$\phi\in \Hom_{Q_W} (Q_{W/W_Q}^*,Q_{W/W_P}^*)$ comes from $\Hom_{\DF} (\mathbb{D}^{\star}_{F,Q},\DPd)$ if and only if the coefficients $b_v \in Q$ satisfy
conditions
\[
b_v'=x_Q b_v \in S\text{ and }x_{v(\al)}\mid (b_v' - b_{\overline{s_{v(\al)}v}}')\text{ for all }\al \notin \Sigma_P.
\]
\end{lem}

\section{Appendix (type $B$)}

We fix subsets $\Theta_P$ and $\Theta_Q$ of the set of simple roots $\Theta$ of the root system $\Sigma$ of type $B$.
We denote the coordinates of vectors in the 
vector space $\R^n$ containing $\Theta$ by $x_1, \ldots, x_n$.
As for the types $A$ and $D$, we denote
$\Theta\setminus (\Theta_Q\cup\{\al_n\})=\{\al_{k_1},\ldots, \al_{k_r}\}$.
If $\al_n \in \Theta_Q$, then denote $\tau=(x_{k_r+1},\ldots, x_n)$, 
otherwise denote the empty list of coordinates by $\tau$.
For simplicity of notation, denote $k_0 = 0$ and $k_{r+1} = n$.
We refer to the tuple $(x_{k_m+1},\ldots,x_{k_{m+1}})$, $0\le m \le r$, as to the $m$-th tuple.
Again we fix a dominant weight $\theta$ with stabilizer $W_P$.

The Weyl group $W_Q$ acts on $\R^n$ by arbitrary permutations of coordinates in each of the tuples 
\[(x_1,\ldots,x_{k_1}),\; (x_{k_1+1},\ldots,x_{k_2}),\; \ldots,\; (x_{k_{r-1}+1},\ldots,x_{k_r}),\; (x_{k_r+1},\ldots, x_n),\] 
and by sign changes of any coordinate in the tuple $\tau$.

For an arbitrary $\mu \in W\theta$, $\mu=(x_1, \ldots, x_n)$, 
and $\al \in \Sigma^+$
such that $s_\al \mu \notin W_Q \mu$,
we will distinguish between the following cases:
\begin{itemize}
\item[(1)] $s_\al=(ij)$, $i<j$. Then $x_{i}$, $x_{j}$ belong to different tuples, $x_{i}\neq x_{j}$, and we have two sub-cases:
\begin{itemize}
\item[(a)] $x_{j}$ does not belong to $\tau$.
\item[(b)] $x_{j}$ belongs to $\tau$.
\end{itemize}

\item[(2)] $s_\al=\widetilde{(ij)}$, $i<j$. Then $x_{i}\neq -x_{j}$ and we have three sub-cases:
\begin{itemize}
\item[(a)] $x_{i}$ and $x_{j}$ belong to different tuples and $x_{j}$ does not belong to $\tau$.
\item[(b)] $x_{i}$ and $x_{j}$ belong to different tuples and $x_{j}$ belongs to $\tau$.
\item[(c)] $x_{i}$ and $x_{j}$ belong to the same tuple, which is not $\tau$.
\end{itemize}

\item[(3)] $s_\al=\widetilde{(i')}$, $x_{i'}\neq 0$, and $x_{i'}$ does not belong to $\tau$.
\end{itemize}

\begin{dfn}
For an arbitary $\mu \in W\theta$, $\mu = (x_1, \ldots, x_n)$, $a \in \R$ and $m \in \Z$ 
($0\le m\le r$), we say that
the number $a$ 
has \emph{multiplicity $\ell$ in the $m$-th tuple of $\mu$} if the number $a$ occurs exactly $\ell$ times 
among $(x_{k_m+1},x_{k_m+2},\ldots,x_{k_{m+1}})$.

We say that
a non-negative number $a \in \R$ has 
\emph{unsigned multiplicity $\ell$ in the tuple $\tau$ of $\mu$}
if the numbers $a$ and $-a$ together 
occur exactly $\ell$ times 
among the coordinates of $\mu$ that belong to $\tau$.
\end{dfn}

\begin{lem}
Suppose $\mu, \mu' \in W\theta$. Then $\mu' \in W_Q \mu$ if and only if the following two conditions hold:
\begin{itemize}
\item For each $m$ ($0\le m \le r$) such that the $m$-th tuple is not $\tau$ and for each $a \in \R$, 
the multiplicities of $a$ in the $m$-th tuples of $\mu$ and of $\mu'$ coincide.
\item For each non-negative $a \in \R$, 
the unsigned multiplicities of $a$ in the tuple $\tau$ of $\mu$ and in the tuple $\tau$ of $\mu'$ coincide.
\end{itemize}
\end{lem}

\begin{proof}
The multiplicities in the tuples other than $\tau$ as well as  the unsigned multiplicities in $\tau$ don't change when we apply elements of $W_Q$. 
On the other hand, using the action of $W_Q$ we can make the sequence of coordinates in each tuple 
increasing, moreover, all numbers in $\tau$ can be made positive. If all multiplicities in the tuples other than $\tau$ 
of $\mu$ and $\tau$ of $\mu'$, as well as all unsigned multiplicities in $\tau$, coincide, then this procedure will bring them to 
the same element of $W\theta$.
\end{proof}

\begin{dfn}\label{dfn:invdesc}
Let $\mu \in W\theta$. 
We say that the \emph{invariant description} of $\mu$ is
the function $d_{\mu} \colon \{0, \ldots, r\} \times \R \to \Z_{\ge 0}$ defined by
\[
d_{\mu} (m, a) = 
\begin{cases}
\text{\begin{tabular}{l} the multiplicity of $a$ in the $m$-th tuple of $\mu$, \\ $\qquad$ if the $m$th tuple is not $\tau$;\end{tabular}}\\
\text{\begin{tabular}{l} the unsigned multiplicity of $a$ in the $m$-th tuple of $\mu$, \\ $\qquad$ if the $m$-th tuple is $\tau$ and $a \ge 0$; \end{tabular}}\\
\text{ 0, if the $m$-th tuple is $\tau$ and $a < 0$.}
\end{cases}
\]
\end{dfn}

\begin{cor}\label{cor:typeBinvariantdescription}
Suppose $\mu, \mu' \in W\theta$. Then $\mu' \in W_Q \mu$ if and only if their invariant descriptions coincide.
\end{cor}

If the case~(1) or the case~(2) holds for $\mu$ and $\al$, we denote by $i$ and $j$ ($i < j$) the indices such that 
$s_\al=(ij)$ or $s_\al=\widetilde{(ij)}$, respectively. 
Suppose that $x_i$ (resp.~$x_j$) belongs to 
the $p$-th (resp.~$q$-th) tuple.

If the case~(3) holds for $\mu$ and $\al$, 
we denote by $i$ the index such that $s_\al=\widetilde{(i)}$ and suppose that $x_i$ belongs to the $p$-th tuple. 

Finally, we denote $f=d_{s_\al \mu}-d_{\mu}$. 

\begin{lem}\label{lem:typeBinvariantdescription-explicit} 
Depending on the cases (1), (2), (3) for $\mu$ and $\al$ we have
\begin{itemize}
\item[(1)]
\begin{itemize}
\item[(a)] $f(p, x_j) = f(q, x_i) = 1,\; f(p, x_i) = f(q, x_j) = -1$;
\item[(b)] $f(p, x_j) = f(q, |x_i|) = 1,\; f(p, x_i) = f(q, |x_j|) = -1$,  if $x_i \ne -x_j$; 

$f(p, x_j) = 1,\; f(p, x_i) = -1$, if  $x_i = -x_j$;
\end{itemize}
\item[(2)]
\begin{itemize}
\item[(a)] 
$f(p, -x_j) = f(q, -x_i) = 1,\; f(p, x_i) = f(q, x_j) = -1$;
\item[(b)] 
$f(p, -x_j) = f(q, |x_i|) = 1,\; f(p, x_i) = f(q, |x_j|) = -1$, if $x_i \ne x_j$;

$f(p, -x_j) = 1,\; f(p, x_i) = -1$, if $x_i = x_j$;
\item[(c)] 
$f(p, -x_j) = f(p, -x_i) = 1,\; f(p, x_i) = f(p, x_j) = -1$, if  $x_i \ne x_j$, $x_i\ne 0$, $x_j\ne 0$;

$f(p, -x_i) = 2,\; f(p, x_i) = -2$, if $x_i = x_j$; 

$f(p, -x_j) = 1,\; f(p, x_j) = -1$, if $x_i \ne x_j$, $x_i = 0$, $x_j\ne 0$;

$f(p, -x_i) = 1,\; f(p, x_i) = -1$, if $x_i \ne x_j$, $x_i\ne 0$, $x_j=0$; 
\end{itemize}
\item[(3)] $f(p, -x_i) = 1,\; f(p, x_i) = -1$;
\end{itemize}
and $f$ equals zero everywhere else.
\end{lem}

\begin{proof}
The statement can be obtained by directly applying Definition \ref{dfn:invdesc} 
to each of the cases (1)--(3). Let us check it in details for the case (1a), all other cases are similar.

Recall that the case (1a) for $\mu$ and $\al$ means that
$s_\al=(ij)$, where $i<j$, the coordinates $x_{i}$, $x_{j}$ belong to different tuples, $x_j$ does not belong to $\tau$, 
and $x_{i}\neq x_{j}$.
We also agreed in the lemma statement that $x_i$ (resp. $x_j$) belongs to the $p$-th (resp. to the $q$-th) tuple.

Denote $s_{\al}\mu = (y_1, \ldots, y_n)$. Then $y_i = x_j$, $y_j = x_i$, and $y_h = x_h$ for all $h \ne i, j$.
In particular, if $m \ne p, q$, then $(y_{k_m + 1}, \ldots, y_{k_{m+1}}) = (x_{k_m + 1}, \ldots, x_{k_{m+1}})$, 
so for each $a \in \R$ we have $d_{s_{\al} \mu}(m, a) = d_{\mu} (m, a)$, and $f(m, a) = 0$. 
Also, if $a \in \R$, $a \ne x_i, x_j$, then the number of coordinates that equal $a$ among $y_{k_p + 1}, \ldots, y_{k_{p+1}}$
is the same as the number of coordinates that equal $a$ among $x_{k_p + 1}, \ldots, x_{k_{p+1}}$, 
because the only different coordinates here are $y_i \ne x_i$, but $a \ne x_i, x_j$. 
So, $d_{s_{\al} \mu}(p, a) = d_{\mu} (p, a)$ and $f(p, a) = 0$. Similarly, $f(q, a) = 0$.

Finally, when we replace $x_i$ with $y_i$ in the $p$th tuple 
(and get $(y_{k_p + 1}, \ldots, y_{k_{p+1}})$ out of $(x_{k_p + 1}, \ldots, x_{k_{p+1}})$ ), 
the number of coordinates that are equal to $x_i$ decreases by 1, and the number of coordinates that are equal to $y_i$ increases by 1.
In other words, $d_{s_{\al} \mu}(p, x_i) = d_{\mu} (p, x_i) - 1$ and $d_{s_{\al} \mu}(p, x_j) = d_{\mu} (p, x_j) + 1$
(recall that $y_i = x_j$), so $f(p, x_i) = -1$ and $f(p, x_j) = 1$.
Similarly, when we replace $x_j$ with $y_j$ in the $q$-th tuple 
to get $(y_{k_q + 1}, \ldots, y_{k_{q+1}})$ out of $(x_{k_q + 1}, \ldots, x_{k_{q+1}})$, 
the number of coordinates that are equal to $x_j$ (resp. $x_i$) decreases (resp. increases) by 1.
In other words, $d_{s_{\al} \mu}(q, x_j) = d_{\mu} (q, x_j) - 1$ and $d_{s_{\al} \mu}(q, x_i) = d_{\mu} (q, x_i) + 1$, 
so $f(q, x_j) = -1$ and $f(q, x_i) = 1$.
\end{proof}

\begin{lem}\label{lem:typeBcase1a2a}
Let $\mu=(x_1,\ldots, x_n) \in W\theta$ and let $\al, \be \in \Sigma^+$ be such that $s_\al \mu$, $s_\be \mu \notin W_Q \mu$.
Suppose also that $s_\be \mu \in W_Q s_\al \mu$. Then
{\begin{enumerate}
\item The case~(1a) holds for $\mu$ and $\al$ if and only if the case~(1a) holds for $\mu$ and $\be$.
In this case, if $s_\al=(i'j')$ and $s_\be=(i''j'')$, where $i' < j'$, $i'' < j''$,
then $i'$ and $i''$ belong to the same tuple, 
$j'$ and $j''$ belong to the same tuple, 
$x_{i'}=x_{i''}$, and $x_{j'}=x_{j''}$.
\item The case~(2a) holds for $\mu$ and $\al$ if and only if the case~(2a) holds for $\mu$ and $\be$.
In this case, if $s_\al=\widetilde{(i'j')}$ and $s_\be=\widetilde{(i''j'')}$, where $i' < j'$, $i'' < j''$,
then $i'$ and $i''$ belong to the same tuple, 
$j'$ and $j''$ belong to the same tuple, 
$x_{i'}=x_{i''}$, and $x_{j'}=x_{j''}$.
\end{enumerate}}
\end{lem}

\begin{proof}
Consider the functions $d_{s_{\al} \mu} - d_\mu$ and $d_{s_{\be} \mu} - d_\mu$.
Since $s_\be \mu \in W_Q s_\al \mu$, they are actually the same function by Corollary \ref{cor:typeBinvariantdescription}.
Denote $f=d_{s_{\al} \mu} - d_\mu=d_{s_{\be} \mu} - d_\mu$.

It follows directly from Lemma \ref{lem:typeBinvariantdescription-explicit} that if the case~(1a) or the case~(2a) holds for $\mu$ and $\al$, 
then there are two numbers $p \in \{0, \ldots, m\}$ such that the $p$-th tuple is not $\tau$
and $f(p, \cdot)$ is a nonzero function (in one variable). 
While if the case~(1b), (2b), (2c), or (3) 
holds for $\mu$ and $\al$, 
then there is only one such number $p \in \{0, \ldots, m\}$ that the $p$-th tuple is not $\tau$
and $f(p, \cdot)$ is a nonzero function (in one variable). 
But $f$ also equals $d_{s_{\be} \mu} - d_\mu$, so by Lemma ~\ref{lem:typeBinvariantdescription-explicit}
applied to $\mu$ and $\be$, we get that the case~(1a) or (2a) holds for $\mu$ and $\be$
if and only if 
there are two numbers $p \in \{0, \ldots, m\}$ such that the $p$-th tuple is not $\tau$
and $f(p, \cdot)$ is a nonzero function (in one variable). 
Therefore, the case~(1a) or (2a) holds for $\mu$ and $\al$ if and only if the case~(1a) or (2a) holds for $\mu$ and $\be$
(although we didn't prove that it is exactly the same case yet).

Until the end of the proof, suppose that 
the case~(1a) or (2a) holds for $\mu$ and $\al$ and the case~(1a) or (2a) holds for $\mu$ and $\be$.
Following the lemma statement, denote by $i'$ and $j'$ ($i' < j'$) the indices such that 
$s_\al=(i'j')$ or $s_\al=\widetilde{(i'j')}$. 
Similarly, denote by $i''$ and $j''$ the indices such that $s_\be=(i''j'')$, $s_\be=\widetilde{(i''j'')}$.

Denote by $p'$ (resp. $p''$, $q'$, $q''$) the indices such 
that $x_{i'}$ (resp. $x_{i''}$, $x_{j'}$, $x_{j''}$) belongs to 
the $p'$th tuple (resp. to the $p''$th tuple, the $q'$th tuple the $q''$th tuple). 
Since $i' < j'$ and $i'' < j''$, we also have $p' < q'$ and $p'' < q''$.

Whatever case holds for $\mu$ and $\al$, we can always 
say that $\ell = p'$ and $\ell = q'$ are exactly the numbers
such that $f(\ell, \cdot)$ is a nonzero function.
And whatever case holds for $\mu$ and $\be$, we can 
say that $\ell = p''$ and $\ell = q''$ are exactly the numbers
such that $f(\ell, \cdot)$ is a nonzero function. 
We also have $p' < q'$ and $p'' < q''$. So, $p'=p''$ and $q'=q''$.

If the case~(1a) holds for $\mu$ and $\al$, then there are two real numbers $a$, $a=x_{i'}$ and $a=x_{j'}$
such that the function $f(\cdot, a)$ takes value 1 once and takes value $-1$ once. For all other $a \in \R$, 
$f(\cdot, a)$ is the zero function. 

If the case~(2a) holds for $\mu$ and $\al$ and $x_{i'} \ne 0$, then $x_{i'}\ne -x_{i'}$, 
and in~(2) we always have $x_{i'}\ne -x_{j'}$, so 
$f(\cdot, x_{i'})$
takes value $-1$ at least once and never takes value $1$.

Finally, if the case~(2a) holds for $\mu$ and $\al$ and $x_{i'} = 0$, then, 
since we always have $x_{i'}\ne -x_{j'}$ in~(2), we can say that $x_{j'}\ne 0$ and 
$x_{j'}\ne -x_{j'}$, so 
$f(\cdot, x_{j'})$
takes value $-1$ at least once and never takes value $1$.
Therefore, the case~(2a) holds for $\mu$ and $\al$ if and only if there exists 
$a \in \R$ such that 
$f(\cdot, a)$
takes value $-1$ at least once and never takes value $1$.
Similarly, the case~(2a) holds for $\mu$ and $\be$ also if and only if there exists 
$a \in \R$ such that 
$f(\cdot, a)$
takes value $-1$ at least once and never takes value $1$.
So, the case~(2a) holds for $\mu$ and $\al$ if and only if the case~(2a) holds for $\mu$ and $\be$.

Now, again regardless  of the case that holds for both $\mu$ and $\al$ and $\mu$ and $\be$, 
by Lemma~\ref{lem:typeBinvariantdescription-explicit} for $\mu$ and $\al$,
$a=x_{i'}$ is the only $a \in \R$ such that $f(p', a)=-1$.
We already know that $p'=p''$, so, 
using Lemma~\ref{lem:typeBinvariantdescription-explicit} for $\mu$ and $\be$,
we can say that $a=x_{i''}$ is the only $a \in \R$ such that $f(p'', a)=-1$.
So, $x_{i'}=x_{i''}$. 
Similarly, using Lemma~\ref{lem:typeBinvariantdescription-explicit},
the function $f(q',\cdot)$ (more specifically, the equation $f(q',a)=-1$ in $a$), 
and the equality $q'=q''$, 
we can say that $x_{j'}=x_{j''}$.
\end{proof}

\begin{lem}\label{lem:typeBbothPQtypeB}
Suppose that 
$\theta=(a,a-1, \ldots, 1, 0, \ldots, 0)$, where $a$ is a non-negative integer and the number of zero coordinates 
in the end is at least $1$,
$k_1=1$, $k_2=2$, \ldots, $k_r=r$, and $\tau=(x_{r+1}, \ldots, x_n)$ ($0 \le r < n$).

Let $\mu= (x_1,\ldots, x_n) \in W\theta$ and $\al, \be \in \Sigma^+$ 
be such that $s_\al \mu, s_\be \mu \notin W_Q \mu$.
Suppose also that $s_\be \mu \in W_Q s_\al \mu$. Then
\begin{enumerate}
\item 
the case~(1b) or the case~(2b) holds for $\mu$ and $\al$ if and only if the case~(1b) or the case~(2b) 
(but not necessarily exactly the same case) holds for $\mu$ and~$\be$. 
\item 
the case~(2c) is impossible for $\mu$ and $\al$, neither it is possible for $\mu$ and $\be$.
\item
the case ~(3) holds for $\mu$ and $\al$ if and only if the case~(3)
holds for $\mu$ and $\be$.
\end{enumerate}
\end{lem}
\begin{proof}
The case~(2c) is impossible since each tuple except $\tau$ has only one coordinate inside.
By Lemma~\ref{lem:typeBcase1a2a}, without loss of generality we may suppose 
that neither the case~(1a) nor the case~(2a) holds for $\mu$ and $\al$ or for $\mu$ and $\be$.

Again, if the cases~(1b) or~(2b) hold for $\mu$ and $\al$ (resp. $\mu$ and $\be$), denote by $i'$ and $j'$ ($i' < j'$) 
(resp. $i''$ and $j''$, $i'' < j''$)
the indices such that 
$s_\al=(i'j')$ or $s_\al=\widetilde{(i'j')}$ (resp. $s_\al=(i''j'')$ or $s_\al=\widetilde{(i''j'')}$). 
If the case~(3) holds for $\mu$ and $\al$ (resp. for $\mu$ and $\be$), 
denote by $i'$ (resp. by $i''$) 
the index such that $s_\al=\widetilde{(i')}$ (resp. $s_\be=\widetilde{(i'')}$).
Again consider the function $f=d_{s_\al}\mu - d_\mu=d_{s_\be \mu} - d_\mu$
(this is the same function by Corollary \ref{cor:typeBinvariantdescription} since $s_\be \mu \in W_Q s_\al \mu$).

Note that each nonzero value appears only once among all of the coordinates 
of $\theta$, and therefore among the coordinates of $\mu$. 
So, any equality of the form $x_i=x_j$ means that $x_i=0$ and $x_i=-x_j$.
So, if the case~(1b) or the case~(2b) holds for $\mu$ and $\al$, 
we always have both inequalities $x_{i'}\ne x_{j'}$ and $x_{i'} \ne -x_{j'}$, 
regardless of whether we actually have the case~(1b) or the case~(2b).
Now Lemma \ref{lem:typeBinvariantdescription-explicit} says that 
if the case~(1b) or the case~(2b) holds for $\mu$ and $\al$, then $f(r, \cdot)$ is a nonzero function.
While if the case~(3) holds for $\mu$ and $\al$, then $f(r,\cdot)$ is a zero function 
by Lemma~\ref{lem:typeBinvariantdescription-explicit} 
since the $r$-th tuple is $\tau$.

So, we can again say that the case~(3) holds for $\mu$ and $\al$ if and only if $f(r,\cdot)$ is a zero function.
Similarly, the case~(3) holds for $\mu$ and $\be$ if and only if $f(r,\cdot)$ is a zero function.
Therefore, 
the case (3) holds for $\mu$ and $\al$ if and only if the case~(3)
holds for $\mu$ and $\be$.
\end{proof}

\begin{lem}\label{lem:typeBbothPQdontcontaintypeB}
Suppose that 
$\theta$ does not have any zero coordinates and $\tau$ is the empty tuple.
Let $\mu= (x_1,\ldots, x_n) \in W\theta$ and $\al, \be \in \Sigma^+$ 
be such that $s_\al \mu, s_\be \mu \notin W_Q \mu$.
Suppose also that $s_\be \mu \in W_Q s_\al \mu$. Then
\begin{enumerate}
\item 
the cases~(1b) and~(2b) 
are impossible for $\mu$ and $\al$, neither they are possible for $\mu$ and $\be$.
\item 
the case~(2c) holds for $\mu$ and $\al$ if and only if the case~(2c)
holds for $\mu$ and $\be$.
\item
the case~(3) holds for $\mu$ and $\al$ if and only if the case~(3)
holds for $\mu$ and $\be$.
\end{enumerate}
\end{lem}
\begin{proof}
The proof is similar to the proof of the previous lemma. 
The cases~(1b) and~(2b) are impossible since $\tau$ is the empty tuple.

Again consider the function $f=d_{s_\al}\mu - d_\mu=d_{s_\be \mu} - d_\mu$
(again, this is the same function by Corollary~\ref{cor:typeBinvariantdescription}).

Since $\theta$ does not have any zero coordinates, $\mu$ cannot have any zero coordinates either.
Then it follows from Lemma~\ref{lem:typeBinvariantdescription-explicit} applied to $\mu$ and $\al$ 
that if the case~(2c) holds for $\mu$ and $\al$, then the sum of all positive values of $f$ is 2.
And if the case~(3) holds for $\mu$ and $\al$, then the sum of all positive values of $f$ is 1.

So, the case~(2c) holds for $\mu$ and $\al$ if and only if the sum of all positive values of $f$ is 2.
Similarly, the case~(2c) holds for $\mu$ and $\be$ if and only if the sum of all positive values of $f$ is 2.
Therefore, the case~(2c) holds for $\mu$ and $\al$ if and only if the case~(2c)
holds for $\mu$ and $\be$.
\end{proof}

\begin{lem}\label{lem:typeBcasesBC3}
Suppose that one of the following is true:
\begin{itemize}
\item There exist numbers $p$ and $q$ ($1 \le p,q \le n$) such that $\Theta_P=\{\al_p,\ldots, \al_n\}$, 
$\Theta_Q=\{\al_q,\ldots,\al_n\}$.

\item $\al_n \notin \Theta_P$ and $\al_n \notin \Theta_Q$.
\end{itemize}
Let $\mu \in W\theta$ ($\mu = (x_1,\ldots, x_n)$) and $\al, \be \in \Sigma^+$ 
be such that $s_\al \mu, s_\be \mu \notin W_Q \mu$.
Suppose also that $s_\be \mu \in W_Q s_\al \mu$. Then
\begin{enumerate}
\item the case~(1b) or the case~(2b) holds for $\al$ if and only if the case~(1b) or the case~(2b) 
(but not necessarily exactly the same case) holds for $\be$.
\item the case~(2c) holds for $\al$ if and only if the case~(2c) 
holds for $\be$. Moreover, then 
$\mu$ does not have any zero coordinates.
\item the case~(3) holds for $s_\al$ if and only if the case~(3) 
holds for $s_\be$.
\end{enumerate}
\end{lem}
\begin{proof}
Suppose that there exist numbers $p$ and $q$ ($1 \le p,q \le n$) such that $\Theta_P=\{\al_p,\ldots, \al_n\}$, 
$\Theta_Q=\{\al_q,\ldots,\al_n\}$. 
Then 
we can take $\theta=(p-1, p-2, \ldots, 1, 0, \ldots, 0)$ (there are $n-p+1$ zeros here).
Also, $r=q-1$, $k_1=1$, $\ldots$, $k_r=q-1$, and $\tau=(x_q, \ldots, x_n)$.
The claim follows from Lemma~\ref{lem:typeBbothPQtypeB} (which also actually says that the case~(2c) is 
impossible for $\al$ or $\be$).

Now suppose that $\al_n \notin \Theta_P$ and $\al_n \notin \Theta_Q$. 
Then $\theta$ does not contain zero coordinates
(if a dominant weight contains a zero coordinate, then the last coordinate is also zero, 
and such a weight is stabilized by $s_{\al_n}$), $\tau$ is the empty tuple, and 
the claim follows from Lemma~\ref{lem:typeBbothPQdontcontaintypeB}.
\end{proof}

\begin{lem}\label{lem:typeBsufficient}
Suppose that one of the following is true:
\begin{itemize}
\item There exist numbers $p$ and $q$ ($1 \le p,q \le n$) such that $\Theta_P=\{\al_p,\ldots, \al_n\}$, 
$\Theta_Q=\{\al_q,\ldots,\al_n\}$.

\item $\al_n \notin \Theta_P$ and $\al_n \notin \Theta_Q$.

\item Either $\Theta_P$ or $\Theta_Q$ coincides with $\Theta$.

\item Either $\Theta_P$ or $\Theta_Q$ is empty.
\end{itemize}
Then 
a double moment graph $\QGP$ is closed.
\end{lem}
\begin{proof}
If $\Theta_P$ or $\Theta_Q$ coincides with $\Theta$ or is empty, then the claim follows from 
Example \ref{ex:pqnothing} and Example~\ref{lem:peverything}. 
Further we suppose that one of the first two conditions in the lemma statement is true, namely:
\begin{itemize}
\item There exist numbers $p$ and $q$ ($1 \le p,q \le n$) such that $\Theta_P=\{\al_p,\ldots, \al_n\}$, 
$\Theta_Q=\{\al_q,\ldots,\al_n\}$.

\item $\al_n \notin \Theta_P$ and $\al_n \notin \Theta_Q$.
\end{itemize}

We are going to use Corollary~\ref{lem:clos}.
Suppose that we have $\mu \in (W\theta)_Q$, $\mu=(x_1, \ldots, x_n)$, and 
two outgoing edges $\mu \to s_{\al}\mu$ and $\mu \to s_{\be}\mu$ in $\GP$ 
such that $s_{\be}\mu \in W_Q s_{\al}\mu$, but $s_{\al}\mu, s_{\be}\mu \notin W_Q \mu$
(actually, in the argument below we don't need the fact that they are outgoing, we will 
only use the fact that $s_{\be}\mu \in W_Q s_{\al}\mu$, but $s_{\al}\mu, s_{\be}\mu \notin W_Q \mu$).
We have to check that there exists 
$w\in W_Q$ such that $w\mu=\mu$ and $\be = w(\al)$.

If the case~(1a) holds for $\mu$ and $\al$ (and $s_\al$ can be written as $(i'j')$ for some $i'$, $j'$), 
then by Lemma~\ref{lem:typeBcase1a2a}, 
the case~(1a) also holds for $\mu$ and $\be$, 
and $s_\be$ can be written as $(i''j'')$ for some $i''$, $j''$.
Moreover, Lemma \ref{lem:typeBcase1a2a}
also says that $x_{i'}=x_{i''}$, $i'$ and $i''$ belong to the same tuple,
$x_{j'}=x_{j''}$,
and 
$j'$ and $j''$ also belong to the same tuple.
Then $w=(i'i'')(j'j'')\in W_Q$.
We have $w \mu = \mu$ and $w \al = \be$.

Similarly, if 
the case~(1a) holds for $\mu$ and $\al$ (and $s_\al=\widetilde{(i'j')}$ for some $i'$, $j'$), 
then by Lemma~\ref{lem:typeBcase1a2a}, 
the case (1a)~also holds for $\mu$ and $\be$, 
and $s_\be=\widetilde{(i''j'')}$ for some $i''$, $j''$.
Moreover, Lemma~\ref{lem:typeBcase1a2a}
again says that $x_{i'}=x_{i''}$, $i'$ and $i''$ belong to the same tuple,
$x_{j'}=x_{j''}$,
and 
$j'$ and $j''$ also belong to the same tuple.
Then $w=(i'i'')(j'j'')\in W_Q$.
We again have $w \mu = \mu$ and $w \al = \be$.

To consider the remaining cases, it is again convenient 
to consider the function $f=d_{s_\al}\mu - d_\mu = d_{s_\be}\mu - d_\mu$
(this is the same function by Corollary~\ref{cor:typeBinvariantdescription}).

Also, if the cases~(1) or (2) hold for $\mu$ and $\al$, denote by $i'$ and $j'$ the indices such that 
$s_{\al}=(i'j')$ or $s_{\al}=\widetilde{(i'j')}$ (respectively).
Similarly, if the cases~(1) or (2) holds for $\mu$ and $\be$, 
denote by $i''$ and $j''$ the indices such that 
$s_{\be}=(i''j'')$ or $s_{\be}=\widetilde{(i''j'')}$.
If the case~(3) holds for $\mu$ and $\al$ (resp. $\mu$ and $\be$), denote by $i'$ (resp. $i''$
the index such that $s_{\al}=\widetilde{(i')}$ (resp. $s_{\be}=\widetilde{(i')}$.

Denote by $p'$ and $p''$ (respectively) the indices such that $x_{i'}$ (resp. $x_{i''}$) belongs to 
the $p'$th tuple (resp. to the $p''$th tuple).

If any of the remaining cases ((1b), (2b), (2c), or (3)) holds for $\mu$ and $\al$, then 
by Lemma~\ref{lem:typeBinvariantdescription-explicit}, 
$f(p',\cdot)$ is not the zero function, and 
$m=p'$ is the only number in $\{1, \ldots, r\}$ such that the $m$th tuple is not $\tau$ and $f(m,\cdot)$ is not the zero function.
By~Lemma \ref{lem:typeBcasesBC3}, one of the cases 
(1b), (2b), (2c), or (3) then also holds for $\mu$ and $\be$, 
and 
by Lemma \ref{lem:typeBinvariantdescription-explicit}, 
$f(p'',\cdot)$ is not the zero function. Therefore, $p'=p''$.
In other words, $i'$ and $i''$ always belong to the same tuple.
By Lemma \ref{lem:typeBcasesBC3} again, 
the case~(3) does not hold for $\mu$ and $\al$
if and only if
the case~(3) does not hold for $\mu$ and $\be$.
Then it follows directly from the definition of the cases~(1b), (2b), and (2c)
and the fact that the case~(2c) holds for 
$\mu$ and $\al$
if and only if
the case~(2c) holds for $\mu$ and $\be$, that 
$j'$ and $j''$ are defined simultaneously, and belong to the same tuple if they are defined.

If the case~(1b) holds for both $\al$ and $\be$,
then by Lemma \ref{lem:typeBbothPQdontcontaintypeB} for $\mu$ and $\al$,
the only $a \in \R$ such that $f(p',a)=-1$ is $a=x_{i'}$.
And by Lemma~\ref{lem:typeBbothPQdontcontaintypeB} for $\mu$ and $\be$,
the only $a \in \R$ such that $f(p',a)=-1$ is $a=x_{i''}$. So, $x_{i'}=x_{i''}$.
Similarly, the consideration of equation $f(p',a)=1$
and Lemma~\ref{lem:typeBbothPQdontcontaintypeB}
applied to $\mu$ and $\al$ and to $\mu$ and $\be$ together imply that $x_{j'}=x_{j''}$.
Then we can take $w=(i'i'')(j'j'') \in W_Q$, then $w \mu=\mu$, $w s_\al \mu = s_\be \mu$, and $w \al=\be$.

If the case~(2b) holds for both $\al$ and $\be$,
then we can again consider 
the equations $f(p',a)=-1$ and $f(p',a)=1$. 
By Lemma~\ref{lem:typeBbothPQdontcontaintypeB} (applied twice to the first equation and twice to the second one), 
the first equation implies $x_{i'}=x_{i''}$, and the second equation implies $-x_{j'}=-x_{j''}$, 
so $x_{j'}=x_{j''}$.
Again we can take $w=(i'i'')(j'j'') \in W_Q$, then $w \mu=\mu$, $w s_\al \mu = s_\be \mu$, and $w \al=\be$.

Suppose that the case~(1b) holds for $\al$ and the case~(2b) holds for $\be$. 
(The situation when the case~(2b) holds for $\al$ and the case~(1b) holds for $\be$ is symmetric.)
This situation is similar to the situation in Type $D$ when 
the case~(1b) holds for $\al$ and the case~(2b) holds for $\be$.

More precisely, using invariant descriptions, we can say that 
the only solution of the equation $f(p',a)=-1$ in $a$ is $x_{i'}$ and $x_{i''}$
at the same time. 
And the only solution of the equation $f(p',a)=1$ in $a$ is $x_{j'}$ and $-x_{j''}$
at the same time. 
So, $x_{i'}=x_{i''}$, and $x_{j'}=-x_{j''}$.
We can take $w=(i'i'')\widetilde{(j'j'')}\in W_Q$ (both $j'$ and $j''$ are in $\tau$), 
then $w \mu=\mu$, $w s_\al \mu = s_\be \mu$, and $w \al=\be$.

If the case~(2c) holds for $\mu$ and $\al$ then it also holds for 
$\mu$ and $\be$ by Lemma~\ref{lem:typeBcasesBC3}. 
Moreover, Lemma~\ref{lem:typeBcasesBC3} also says that 
$\mu$ does not have any zero coordinates. 
Then, by Lemma~\ref{lem:typeBinvariantdescription-explicit} for $\mu$ and $\al$, 
the numbers $a \in \R$ such that $f(p', a)$ is positive 
are exactly $a=x_{i'}$ and $a=x_{j'}$.
By Lemma~\ref{lem:typeBinvariantdescription-explicit} for $\mu$ and $\be$, 
the solutions of $f(p',a)>0$ in $a$ are $a=x_{i''}$ and $a=x_{j''}$.
So, $\{x_{i'}, x_{j'}\} = \{x_{i''}, x_{j''}\}$, 
and at least one of the following is true:
\begin{enumerate}
\item $x_{i'}=x_{i''}$ and $x_{j'}=x_{j''}$. Take $w=(i'i'')(j'j'')$.
\item $x_{i'}=x_{j''}$ and $x_{j'}=x_{i''}$. Take $w=(i'j'')(j'i'')$.
\end{enumerate}
In both cases, $w \mu=\mu$, $w s_\al \mu = s_\be \mu$, and $w \al=\be$.

Finally, if the case~(3) holds for both $s_\al$ and $s_\be$,
then the equation $f(p',a)=1$ and Lemma \ref{lem:typeBinvariantdescription-explicit} applied to $\mu$ and $\al$
and to $\mu$ and $\be$ 
show that $x_{i'}=x_{i''}$.
So, we can take $w=(i'i'')$, then $w \mu=\mu$, $w s_\al \mu = s_\be \mu$, and $w \al=\be$.

Therefore, by Corollary~\ref{lem:clos}, 
the double moment graph $\QGP$ is closed.
\end{proof}

\begin{lem}\label{lem:typeBncessary}
Suppose that $\Theta_P \ne \Theta$, $\Theta_P \ne \emptyset$, $\Theta_Q \ne \Theta$, $\Theta_Q \ne \emptyset$, 
and
one of the following conditions is true:
\begin{enumerate}
\item $\al_n \in \Theta_P$, 
and $\Theta_Q$ is not of the form $\{\al_q, \ldots, \al_n\}$ for any $q$ ($1 \le q \le n$).
\item $\al_n \in \Theta_Q$, 
and $\Theta_P$ is not of the form $\{\al_p, \ldots, \al_n\}$ for any $p$ ($1 \le p \le n$).
\end{enumerate}
Then
a double moment graph $\QGP$ is never closed.
\end{lem}

\begin{proof}
Suppose that 
$\al_n \in \Theta_P$, 
and $\Theta_Q$ is not of the form $\{\al_q, \ldots, \al_n\}$ for any $q$ ($1 \le q \le n$).
Then there are roots $\al_i$ and $\al_j$ such that $i < j$, $\al_i \in \Theta_Q$, 
$\al_j \notin \Theta_Q$.
Then $\theta$ contains at least one zero coordinate, and 
$x_i$ and $x_{i+1}$ belong to the same tuple of coordinates, which is not $\tau$.
Also $\theta$ has at least one nonzero (actually, positive since this is a dominant weight) 
coordinate since $\Theta_P \ne \Theta$. Suppose that the value of this coordinate is $a > 0$.

Then we can choose $u \in W^P$ so that for $\mu = u \theta$ we have $x_i=a$, $x_{i+1}=0$.
Set $\al=e_i$, $\be=e_i+e_{i+1}$. 
Then $s_\al=\widetilde{(i)}$, 
$s_\be=\widetilde{(i,i+1)}$.
Then $s_\al \mu$ has $x_i=-a$, $x_{i+1}=0$, and all other coordinates are the same as for $\mu$.
And $s_\be \mu$ has $x_i=0$, $x_{i+1}=-a$, and all other coordinates are the same as for $\mu$.
So, $s_\be \mu = (i,i+1)s_\al \mu$,
and $(i,i+1) \in W_Q$, so $s_\be \mu \in W_Q s_\al \mu$, 
but $\be \notin W_Q \al$, because $\al$ is a short root, and $\be$ is a long root.
Also, $(\mu, \al) = (\mu, \be) = a > 0$, so $\mu < s_\al \mu$ and $\mu < s_\be \mu$.
So, by Corollary~\ref{lem:clos}, 
the double moment graph $\QGP$ is not closed.

Now suppose that 
$\al_n \in \Theta_Q$, 
and $\Theta_P$ is not of the form $\{\al_p, \ldots, \al_n\}$ for any $p$ ($1 \le p \le n$).
Then $\tau$ is not the empty tuple, $x_n \in \tau$, and 
there are roots $\al_i$ and $\al_j$ such that $i < j$, $\al_i \in \Theta_P$, 
$\al_j \notin \Theta_P$.
Then the coordinates $x_i$ and $x_{i+1}$ of $\theta$ are equal and nonzero, suppose they equal $a$.
(Again, $a > 0$ since this is a dominant weight.)
Also, $\tau$ is not the only tuple since $\Theta_Q \ne \Theta$, in particular, 
$x_1 \notin \tau$.

Then we can choose $u \in W^P$ so that for $\mu = u \theta$ we have $x_1=a$, $x_n=a$.
Set $\al=e_1$, $\be=e_1+e_n$. 
Then $s_\al=\widetilde{(1)}$, 
$s_\be=\widetilde{(1,n)}$.
Then $s_\al \mu$ has $x_1=-a$, $x_n=a$, and all other coordinates are the same as for $\mu$.
And $s_\be \mu$ has $x_1=-a$, $x_n=-a$, and all other coordinates are the same as for $\mu$.
So, $s_\be \mu = \widetilde{(n)}s_\al \mu$,
and $\widetilde{(n)} \in W_Q$, so $s_\be \mu \in W_Q s_\al \mu$, 
but $\be \notin W_Q \al$, because $\al$ is a short root, and $\be$ is a long root.
Also, $(\mu, \al) = a > 0$, $(\mu, \be) = 2a > 0$, so $\mu < s_\al \mu$ and $\mu < s_\be \mu$.
So, again by Corollary~\ref{lem:clos}, 
the double moment graph $\QGP$ is not closed.
\end{proof}

\begin{proof}[Proof of Proposition~\ref{prop:typeB}]
The ``if'' part follows directly from Lemma \ref{lem:typeBsufficient}. 
For the ``only if'' part, suppose that 
none of the four conditions in the statement of the proposition is true.
In other words, suppose that $\Theta_P \ne \Theta$, $\Theta_Q \ne \Theta$, 
$\Theta_P \ne \emptyset$, $\Theta_Q \ne \emptyset$, and, speaking 
in terms of mathematical logic, the disjunction of the following two conditions is false:
\begin{enumerate}
\item There exist numbers $p$ and $q$ ($1 \le p, q \le n$)
such that 
$\Theta_P=\{\al_p,\ldots, \al_n\}$, 
$\Theta_Q=\{\al_q,\ldots,\al_n\}$.
\item $\al_n \notin \Theta_P$ and $\al_n \notin \Theta_Q$.
\end{enumerate}
Clearly, the statement 
``$\al_n \notin \Theta_P$, and 
there exist a number $p$ ($1 \le p \le n$)
such that 
$\Theta_P=\{\al_p,\ldots, \al_n\}$''
is always false, as well as the similar statement for $\Theta_Q$.
So, we can add these two statements to our disjunction and get an equivalent disjunction.
So, the disjunction of the following four conditions is false (note that we reordered some statements, this will be useful below):
\begin{enumerate}
\item 
(There exists a number $q$ ($1 \le q \le n$)
such that
$\Theta_Q=\{\al_q,\ldots,\al_n\}$),
and
(there exists a number $p$ ($1 \le p \le n$)
such that 
$\Theta_P=\{\al_p,\ldots, \al_n\}$).
\item (There exists a number $q$ ($1 \le q \le n$)
such that 
$\Theta_Q=\{\al_q,\ldots, \al_n\}$), 
and
$\al_n \notin \Theta_Q$.
\item $\al_n \notin \Theta_P$ and 
(there exists a number $p$ ($1 \le p \le n$)
such that
$\Theta_P=\{\al_p,\ldots,\al_n\}$).
\item $\al_n \notin \Theta_P$ and $\al_n \notin \Theta_Q$.
\end{enumerate}
Now, using conjunction-disjunction distributivity, we can factor this disjunction into a conjunction of two disjunctions.
So, the conjunction of the following two conditions is false:
\begin{enumerate}
\item 
(There exists a number $q$ ($1 \le q \le n$)
such that
$\Theta_Q=\{\al_q,\ldots, \al_n\}$)
or
$\al_n \notin \Theta_P$.
\item (There exist a number $p$ ($1 \le p \le n$)
such that 
$\Theta_P=\{\al_p,\ldots, \al_n\}$)
or
$\al_n \notin \Theta_Q$.
\end{enumerate}
Equivalently, the disjunction of the following two conditions is true (i.e. at least one of them is true):

If $\Theta_P$ is of the form $\{\al_p,\ldots, \al_n\}$ for some $p$ ($1\le p\le n$), 
then the first condition above says that 
$\Theta_Q$ is not of the form $\{\al_q,\ldots, \al_n\}$, 
and we also see directly that $\al_n \in \Theta_P$.

Symmetrically, if $\Theta_Q$ is of the form $\{\al_q,\ldots, \al_n\}$, 
then 
$\al_n \in \Theta_Q$, and $\Theta_P$ is not of the form $\{\al_p,\ldots, \al_n\}$.

And if neither $\Theta_P$, nor $\Theta_Q$ is of the form $\{\al_p,\ldots, \al_n\}$
for any $p$ ($1\le p \le n$),
then the second condition above still says that 
$\Theta_P$ or $\Theta_Q$ contains a subsystem of type $B$.

So, one of the following conditions is true:
\begin{enumerate}
\item $\Theta_Q$ is not of the form $\{\al_q,\ldots, \al_n\}$ for any $q$ ($1\le q\le n$), and $\al_n \in \Theta_P$.
\item $\Theta_P$ is not of the form $\{\al_p,\ldots, \al_n\}$ for any $p$ ($1\le p\le n$), and $\al_n \in \Theta_Q$.
\end{enumerate}
The claim follows from Lemma \ref{lem:typeBncessary}.
\end{proof}

\bibliographystyle{alpha}

\end{document}